\documentclass[11pt]{amsart}
\usepackage{graphicx}
\usepackage{tgpagella}
\usepackage{euler}
\usepackage[T1]{fontenc}
\usepackage{amssymb}
\usepackage{amsmath,amsthm}
\usepackage[hidelinks]{hyperref}
\usepackage{amsthm}
\usepackage[english]{babel}
\usepackage{mathrsfs}
\usepackage{eucal}

\newtheorem{main}{Theorem}

\newtheorem{mcor}[main]{Corollary}

\newtheorem{theorem}{Theorem}[section]
\newtheorem{lem}[theorem]{Lemma}
\newtheorem{prop}[theorem]{Proposition}

\newtheorem{cor}[theorem]{Corollary}

\theoremstyle{definition}
\newtheorem{definition}[theorem]{Definition}
\newtheorem{notation}[theorem]{Notation}
\newtheorem{Remark}[theorem]{Remark}

\newtheorem{claim}[theorem]{Claim}
\def\ca{\curvearrowright}
\def\ra{\rightarrow}

\def\la{\lambda}
\def\La{\Lambda}

\def\ve{\varepsilon}

\def\g{\gamma}

\def\G{\Gamma}
\def\Cal{\mathcal}

\numberwithin{equation}{section}

\def\mod{\mathrm{Mod}}
\def\pmod{\mathrm{PMod}}

\newcommand{\calg}{\mathcal{G}}
\newcommand{\cals}{\mathcal{S}}
\newcommand{\caln}{\mathcal{N}}
\newcommand{\calm}{\mathcal{M}}
\newcommand{\calh}{\mathcal{H}}
\newcommand{\calt}{\mathcal{T}}
\newcommand{\calu}{\mathcal{U}}
\newcommand{\calc}{\mathcal{C}}
\newcommand{\call}{\mathcal{L}}
\newcommand{\calcp}{\mathcal{CP}}
\newcommand{\cali}{\mathcal{I}}
\newcommand{\calk}{\mathcal{K}}
\newcommand{\comm}{\mathrm{Comm}}

\begin{document}

\title[Superrigidity Results for Actions by Surface Braid Groups]{$OE$ and $W^*$ Superrigidity Results for Actions by Surface Braid Groups}
\author[I. Chifan]{Ionut Chifan}
\address{Department of Mathematics, The University of Iowa, 14 MacLean Hall, IA  
52242, USA and IMAR, Bucharest, Romania}
\email{ionut-chifan@uiowa.edu}
\thanks{I.C.\ was supported in part by the Old Gold Fellowship, University of Iowa and by NSF Grant \#1301370.}
\author[Y. Kida]{Yoshikata Kida}
\address{Graduate School of Mathematical Sciences, The University of Tokyo, Komaba, Tokyo 153-8914, Japan}
\email{kida@ms.u-tokyo.ac.jp}
\thanks{Y.K.\ was supported by JSPS Grant-in-Aid for Young Scientists (B), No.25800063}
\date{September 1, 2015}
\dedicatory{}

\begin{abstract}

We show that several important normal subgroups $\G$ of the mapping class group of a surface satisfy the following property: any free, ergodic, probability measure preserving action $\G\ca X$ is stably $OE$-superrigid. These include the central quotients of most surface braid groups and most Torelli groups and Johnson kernels. In addition, we show that all these groups satisfy the measure equivalence rigidity and we describe all their lattice-embeddings.


Using these results in combination with previous results from \cite{CIK13} we deduce that any free, ergodic, probability measure preserving action of almost any surface braid group is stably $W^*$-superrigid, i.e., it can be completely reconstructed from its von Neumann algebra.

\end{abstract}

\maketitle


\section{Introduction}

The study of rigidity phenomena in orbit equivalence has its inception in the fundamental work of Zimmer, who showed that free, ergodic, probability measure preserving (p.m.p.) actions of higher-rank simple Lie groups and their lattices satisfy a strong-type rigidity, \cite{Zi80}. This is a consequence of Zimmer's cocycle superrigidity theorem which in turn extends Margulis' superrigidity theorem in the context of measurable cocycles, \cite{Ma74}. In the last decades extreme forms of rigidity in orbit equivalence have been discovered. Two free, p.m.p.\ actions of countable groups $\G\ca X$ and $\La \ca Y$ are called \emph{orbit equivalent} if there exists a probability measure spaces isomorphism $\psi \colon X\ra Y$ such that $\psi(\G x)=\La \psi(x)$, for almost every $x\in X$. A free, ergodic, p.m.p.\ action of a countable group $\Gamma \ca X$  is called \textit{(stably) $OE$-superrigid} if it satisfies the property that any free, ergodic, p.m.p.\ action of a countable group $\Lambda \ca Y$, which is (stably) orbit equivalent to $\Gamma \ca X$, is (stably) conjugate to $\G \ca X$. In other words, these are actions which are completely reconstructible from their orbits. The first such examples emerged from the groundbreaking work of Furman, who showed that the action of $SL_n(\mathbb{Z})$ on the $n$-torus with $n\geq 3$ is $OE$-superrigid, \cite{Fu99a, Fu99b}. Since then, many other examples have been discovered through combined efforts spawning from several directions of research including measurable methods group theory, geometric group theory, Popa's deformation/rigidity theory, or $C^*$-algebraic techniques, \cite{BFS10, Fu06, Io08, Io14, Ki08, Ki09b, Ki10, MS04, Po05, Po08, PV08, Sa09}. For further references and other related topics in orbit equivalence we encourage the reader to consult the following excellent surveys \cite{Fu09, Ga09, Po06, Sh05, V10}.

Let $S=S_{g, k}$ be a connected, compact, orientable surface of genus $g$ with $k$ boundary components and denote by $\mod(S)$ the corresponding mapping class group. 
The second author showed that, given any high complexity surface $S$ ($3g+k-4>0$), any free, ergodic, p.m.p.\ action $\mod(S)\ca X$ is stably $OE$-superrigid, \cite{Ki06, Ki08}; this was the first occurrence of examples of infinite, countable groups whose \emph{arbitrary} action is stably $OE$-superrigid.
In the subsequent work, the same rigidity was obtained for free, ergodic, p.m.p.\ actions of the amalgamated free product $SL_3(\mathbb{Z})\ast_\Sigma SL_3(\mathbb{Z})$, where $\Sigma$ is the subgroup of $SL_3(\mathbb{Z})$ consisting of matrices $a=(a_{ij})$ with $a_{31}=a_{32}=0$, \cite{Ki09b}.

One of the main goals of this paper is to provide other natural examples of groups which satisfy the superrigidity phenomenon described above. Let $M$ be a closed, orientable surface of genus $g$, let $k$ be a positive integer, and denote by $F_k(M)$ the space of $k$-ordered, mutually distinct points of $M$. The fundamental group of $F_k(M)$ is denoted by $PB_k(M)$ and it is called the \textit{pure braid group of $k$ strands on $M$}. Notice that $PB_1(M)=\pi_1(M)$.
Moreover, using Birman's exact sequence \cite{Bi74}, the group $PB_k(M)$---with only a few exceptions---can be naturally identified with a certain normal subgroup of the mapping class group $\mod(S)$, where $S=S_{g, k}$. This subgroup of $\mod(S)$ will be henceforth denoted by $P(S)$. Building upon previous methods and results from \cite{Ki05, Ki06, KY10a, KY10b}, we show the following:

\begin{main}\label{main1}
Let $M$ be a closed orientable surface of genus $g\geq 2$ and let $k\geq 2$ be an integer. Then any free, ergodic, p.m.p.\ action $PB_k(M)\ca X$ is stably $OE$-superrigid.
\end{main}

When $g=0$ or $1$, the group $PB_k(M)$ has infinite center and the same conclusion holds for its central quotient, assuming $2g+k\geq 5$. In addition, we show that the theorem still holds true if instead of $PB_k(M)\simeq P(S)$ one considers other natural, normal subgroups of $\mod(S)$, such as the \textit{Torelli group}, $\cali(S)$ and the \textit{Johnson kernel}, $\calk(S)$ (see Theorem \ref{oesuperrigidity}) or, more generally, any finite direct product between all the aforementioned groups (see Remark \ref{oerig-prod}).

As a by-product of our methods, we obtain measure equivalence rigidity results for $PB_k(M)$ and we describe all its lattice-embeddings. Measure equivalence is a notion introduced by Gromov \cite{Gr93} as a measure-theoretic counterpart to quasi-isometry between finitely generated groups.
It is well known that two countable groups $\Gamma$ and $\Lambda$ are measure equivalent if and only if there exist free, ergodic, p.m.p.\ actions $\Gamma \ca X$ and $\Lambda \ca Y$ which are stably orbit equivalent,  \cite{Fu99b}. Thus, assertion (i) in the following theorem is a corollary of Theorem \ref{main1}. Assertion (ii) follows by appealing to the methods developed in \cite{Fu01} to describe all lattice-embeddings for higher rank lattices.

\begin{main}\label{main2}
Let $M$ be a closed, orientable surface of genus $g\geq 2$ and let $k\geq 2$ be an integer. Then the following hold:
\begin{enumerate}
\item[(i)] Any countable group that is measure equivalent to $PB_k(M)$ is virtually isomorphic to $PB_k(M)$.
\item[(ii)] Let $\Gamma$ be a finite index subgroup of $PB_k(M)$.
Then any locally compact, second countable group containing a lattice isomorphic to $\Gamma$ admits a finite index subgroup which is continuously isomorphic to a semidirect product $\Gamma \ltimes K$, where $K$ is a compact group. 
\end{enumerate}
\end{main}

As before, the same conclusion holds for the Torelli group and the Johnson kernel of most surfaces, as well as any finite direct product of such groups.
We refer the reader to Section \ref{sec-me-oe} for the precise statements.

\vskip 0.03in

A free, ergodic, p.m.p.\ action $\G \ca X$ is called \emph{(stably) $W^*$-superrigid} if, for any given free, ergodic, p.m.p.\ action $\La \ca  Y$, any isomorphism between the group von Neumann algebras $L^\infty(X)\rtimes \G$ and $L^\infty(Y)\rtimes\La$ entails a (stable) conjugacy between the actions $\G \ca X$ and $ \La\ca Y$. In other words, these are actions which can be completely reconstructed---up to (stable) conjugacy---from their von Neumann algebras. The study of $W^*$-superrigid actions has received a lot of attention over the last years as it plays a central role in the classification of the group measure space von Neumann algebras. It was noted by Singer that two free, ergodic, p.m.p.\ actions are orbit equivalent if and only if  there exists an isomorphism between the corresponding von Neumann algebras preserving the canonical Cartan subalgebras, \cite{Si55}. Thus, a free, ergodic, p.m.p.\  action $\G \ca X$ is (stably) $OE$-superrigid if and only if the following property is satisfied: whenever $ \La\ca Y$  is a free, ergodic, p.m.p.\ action, the existence of a $*$-isomorphism $\psi \colon L^\infty(X)\rtimes \G \ra L^\infty(Y)\rtimes \La$ satisfying $\psi(L^\infty(X))=L^\infty (Y)$ entails that the actions $\G\ca X$ and $\La\ca Y$ must be (stably) conjugate. Consequently, (stable) $W^*$-superrigidity is a priori stronger than (stable) $OE$-superrigidity as it is the logical sum of the former and the property that $L^\infty(X)$ is the unique group measure space Cartan subalgebra of  $L^\infty(X)\rtimes \G$, up to unitary conjugacy. Free, ergodic, p.m.p.\ actions satisfying this last property are called \emph{$\mathcal C$-superrigid}. 

To study the $W^*$-superrigidity phenomenon and other problems of central importance in the classification of von Neumann algebras and related fields, Popa introduced (over a decade ago) a completely new conceptual framework, now termed  \emph{Popa's deformation/rigidity theory}. This theory develops a powerful technical paraphernalia designed to incorporate various cohomological, geometric, and algebraic information of a group and its actions at the von Neumann algebraic level. Through this novel approach Popa and his co-authors obtained striking classification results for the group von Neumann algebras and beyond. For instance, in their remarkable work \cite{OP07}, Ozawa and Popa discovered the first examples of $\mathcal C$-superrigid actions: all free, ergodic, profinite p.m.p.\ actions of non-abelian free groups on standard probability spaces. This is a result that deeply influenced the entire subsequent developments on the classification of group measure space von Neumann algebras.


In \cite{Pe09}, Peterson proved the existence of $W^*$-superrigid actions. Shortly after, Popa and Vaes were able to provide the first concrete examples $W^*$-superrigid actions \cite{PV09}. Their examples include large classes of Bernoulli actions of various amalgamated free products groups. Later, Ioana managed to show that any Bernoulli action $\G\ca [0,1]^\G$ is $W^*$-superrigid whenever $\G$ is an ICC property (T) group. Other examples of $W^*$-superrigid actions were unveiled subsequently, through combined efforts spawning from both measurable methods in orbit equivalence and Popa's deformation/rigidity theory, \cite{Po05,Ki06,Ki08,Io08,Pe09,Ki09b,PV09,FV10,CP10,HPV10,Io10,IPV10,Va10,CS11,CSU11,PV11,PV12,Bo12,CIK13}.

In particular, Houdayer-Popa-Vaes discovered that if $\Gamma =SL_3(\mathbb{Z})\ast_\Sigma SL_3(\mathbb{Z})$ then any free, ergodic, p.m.p.\ action $\Gamma \ca X$ is $\mathcal C$-superrigid, \cite{HPV10}. Combining this with the second author's stable $OE$-superrigidity result from \cite{Ki09b} it follows that any free, ergodic, p.m.p.\ action $\Gamma \ca X$ is stably $W^*$-superrigid. This was the first instance of a countable infinite group whose \emph{every} free, ergodic, p.m.p.\ action is stably $W^*$-superrigid.  Ioana and the authors proved in \cite{CIK13} that all mapping class groups $\mod(S_{g,k})$ of high complexity surfaces $S_{g,k}$ of low genus ($g\leq 2$) are other examples of such groups. This was essentially obtained from the second author's previous results on stable $OE$-superrigidity for mapping class groups \cite{Ki06} and the $\mathcal C$-superrigidity results for actions by finite step extensions of hyperbolic groups, \cite[Corollary 3.2]{CIK13} (see also \cite[Proposition 4.6]{VV14}). 

In this paper we show that almost all surface braid groups $PB_k(M)$ satisfy the superrigidity phenomenon described above, thus adding numerous new examples to the list. Indeed, by \cite[Theorem 3.7]{CKP14}, most of these groups are finite step extensions by hyperbolic groups and, by  \cite[Corollary 3.2]{CIK13}, all their free, ergodic, p.m.p.\ actions are $\mathcal C$-superrigid. This, in combination with Theorem \ref{main1}, leads to the following:

\begin{mcor}\label{main3}
Let $M$ be a closed orientable surface of genus $g\geq 2$ and let $k\geq 2$ be an integer. Then any free, ergodic, p.m.p.\ action $PB_k(M)\ca X$ is stably $W^*$-superrigid.
\end{mcor}

Moreover, we show the same result holds for most Torelli groups and Johnson kernels associated with surfaces of genus one (Theorem \ref{wsuperrigid}). As a result, Remark \ref{identification} combined with Corollary \ref{main3} and  \cite[Theorem A]{CIK13} settle completely the $W^*$-superrigidity question for all free, ergodic, p.m.p.\ actions of the central quotients of surface braid groups. 

In the last part of the paper we exhibit a family of subgroups in direct products of iterated amalgams or HNN-extensions over abelian subgroups whose all free, mixing, p.m.p.\ actions are $\mathcal C$-superrigid (see Theorem \ref{c-sr1} for the precise statement). The result is obtained by combining powerful, recent classification results on normalizers of subalgebras in II$_1$ factors from \cite{PV12,Io12,Va13} with techniques from \cite{CIK13}.  

When this is combined with recent developments in group theory \cite{KM13,KM12} we obtain the following:

\begin{mcor}\label{main4}  Let $\Lambda$ be a torsion-free, non-elementary hyperbolic group. If $\G$ is any finitely presented, ICC, residually-$\Lambda$ group, then  any free, mixing, p.m.p.\ action $\G \ca X$ on a non-atomic probability space is $\mathcal C$-superrigid.
\end{mcor}



When $\G$ is limit group the result is already known from \cite[Theorem 1.1]{PV12}. When $\G$ is an ICC, residually free group then the result still holds if we only assume that $\G$ is finitely generated as opposed to finitely presented. This follows because results in \cite{CG05,KM98a,KM98b,Se01} still allow one to embed such groups in a finite direct product of iterated amalgams over abelian subgroups and Theorem \ref{c-sr1} applies.

The $W^*$-superrigidity question for the group actions described in Corollary \ref{main4} remains open as we do not have a complete understanding of an additional, general condition on these groups and their actions to insure the $OE$-superrigidity part. For instance, there exist such group actions which are $OE$-superrigid (e.g., any Bernoulli action $\G\ca [0,1]^\G$, where $\G=\mathbb F_2\times \mathbb F_2$, \cite{Po08}) but there also exist such group actions which are not $OE$-superrigid (e.g.,  any Bernoulli action $\G\ca [0,1]^\G$, where $\G=\mathbb F_2$). However, we conjecture the $OE$-superrigidity statement holds under the additional assumption that $\G$ is not fully residually-$\Lambda$ (Section \ref{prelimrhg}).

\subsection{Comments on the proofs} Our strategy to show $OE$-superrigidity follows closely the methods from \cite{Ki06}, where it was first shown that $\mod(S)$ satisfies the following strong-type rigidity: for any two free, ergodic, p.m.p.\ actions of $\mod(S)$ that are orbit equivalent, the corresponding cocycle is equivalent to a virtual automorphism of $\mod(S)$. Once this is established, the $OE$-superrigidity for arbitrary actions of $\mod (S)$ is deduced via Furman's general machinery \cite{Fu99a, Fu99b}. This general strategy was successfully applied to establish many  subsequent $OE$-superrigidity results, \cite{BFS10, Ki09b, Ki10, MS04, Sa09}.

A large part of our paper is devoted to show strong-type rigidity for groups $P(S)$, $\cali(S)$ and $\calk(S)$. Given two free, ergodic, p.m.p.\ actions $P(S)\ca X$, $P(S)\ca Y$ together with an isomorphism between the associated groupoids, $\phi\colon P(S)\ltimes X\simeq P(S)\ltimes Y$, we will analyze geometric subgroupoids and how they are preserved under $\phi$. To understand algebraic properties of such groupoids, we will exploit their classification from \cite{Ki05, Ki06} which in essence parallels McCarthy and Papadopoulos' classification of subgroups of $\mod(S)$, \cite{MP89}. A key ingredient to conclude that $\phi$ is equivalent to a conjugacy between $P(S)\ca X$ and $P(S)\ca Y$ is the complex of hole-bounding curves and pairs of $S$, denoted by $\calcp(S)$. This is a version of the complex of curves of $S$ that was introduced in \cite{KY10a} to compute virtual automorphisms of $P(S)$, inspired by \cite{Iv97, IIM03}. Finally,  to conclude our result we will use the results from \cite{KY10a, KY10b} that describe all simplicial automorphisms of $\calcp(S)$ as well as the structure of a certain injection from a subcomplex of $\calcp(S)$ into $\calcp(S)$.
In the case of $\cali(S)$ and $\calk(S)$, we will use similar results from \cite{BM04, BM08, FI05, Ki09c, KY10c} which are applicable to the corresponding versions of curves complexes. We refer the reader to Subsections \ref{subsec-pre-sbg} and \ref{subsec-cpx-t-j} for more details on these simplicial complexes.

\subsection{Organization of the paper} In Section \ref{sec-pre}, we review the mapping class groups and the classification of their subgroups. 
In Section \ref{sec-groupoids-mcg}, given a subgroup $\Gamma <\mod(S)$ and a p.m.p.\ action $\Gamma \ca X$, we review the classification of subgroupoids of $\Gamma \ltimes X$ developed in \cite{Ki05, Ki06}.
We study a certain chain of subgroupoids to characterize geometric subgroupoids algebraically.
In Section \ref{sec-sbg} we prove the strong-type rigidity for $P(S)$, then Section \ref{sec-t-j} deals with the strong-type rigidity of $\cali(S)$ and $\calk(S)$.
In Section \ref{sec-me-oe}, we discuss several consequences of the strong-type rigidity: OE superrigidity, measure equivalence rigidity, description of lattice-embeddings, and rigidity of direct products. In Section \ref{sec-w}, we present several applications to $W^*$-superrigidity. In the last section, we prove the technical steps leading to the $\mathcal C$-superrigidity of all free, mixing, p.m.p.\ actions by finitely presented, residually hyperbolic groups.

\subsection{Acknowledgments} The authors are very grateful to the anonymous referee for his comments and suggestions which greatly improved the exposition and the overall mathematical quality of the paper.

\section{Preliminaries on mapping class groups}\label{sec-pre}

\subsection{Surfaces and curves}\label{subsec-curve}

Unless otherwise mentioned, we assume a surface to be connected, compact and orientable. 
Let $S=S_{g, k}$ be a surface of genus $g$ with $k$ boundary components.
A simple closed curve in $S$ is called {\it essential} in $S$ if it is neither homotopic to a single point of $S$ nor isotopic to a component of $\partial S$.
When there is no confusion, we mean by a curve in $S$ either an essential simple closed curve in $S$ or its isotopy class. 
A curve $\alpha$ in $S$ is called {\it separating} in $S$ if $S\setminus \alpha$ is not connected.
Otherwise $\alpha$ is called {\it non-separating} in $S$.
Whether $\alpha$ is separating in $S$ or not depends only on the isotopy class of $\alpha$.
A pair of non-separating curves in $S$, $\{ \beta, \gamma \}$, is called a \textit{bounding pair (BP)} in $S$ if $\beta$ and $\gamma$ are disjoint and non-isotopic and $S\setminus (\beta \cup \gamma)$ is not connected.
This condition depends only on the isotopy classes of $\beta$ and $\gamma$ (see Figure \ref{fig-snsbp}).

\begin{figure}
\begin{center}
\includegraphics[width=7cm]{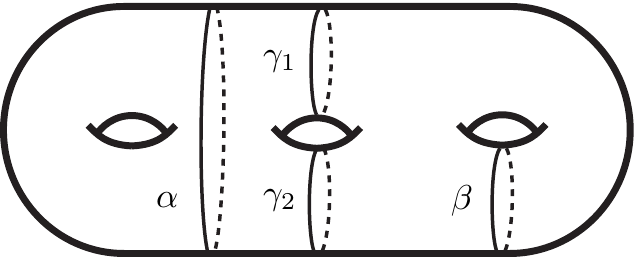}
\caption{The curve $\alpha$ is a separating curve, $\beta$ is a non-separating curve, and $\{ \gamma_1, \gamma_2\}$ is a BP.}\label{fig-snsbp}
\end{center}
\end{figure}

We define $V(S)$ as the set of isotopy classes of essential simple closed curves in $S$. 
We denote by $I\colon V(S)\times V(S)\rightarrow \mathbb{Z}_{\geq 0}$ the {\it geometric intersection number}, i.e., the minimal cardinality of the intersection of representatives for two elements of $V(S)$.
Let $\Sigma(S)$ denote the set of non-empty finite subsets $\sigma$ of $V(S)$ with $I(\alpha, \beta)=0$ for any $\alpha, \beta \in \sigma$.
We extend the function $I$ to the symmetric function on the square of $V(S)\sqcup \Sigma(S)$ with $I(\alpha, \sigma)=\sum_{\beta \in \sigma}I(\alpha, \beta)$ and $I(\sigma, \tau)=\sum_{\beta \in \sigma, \gamma \in \tau}I(\beta, \gamma)$ for any $\alpha \in V(S)$ and $\sigma, \tau \in \Sigma(S)$.
Let $V_s(S)$ denote the subset of $V(S)$ consisting of isotopy classes of separating curves in $S$.
Let $V_{bp}(S)$ denote the subset of $\Sigma(S)$ consisting of isotopy classes of BPs in $S$.

Let $\sigma$ be an element of $\Sigma(S)$.
We denote by $S_\sigma$ the surface obtained by cutting $S$ along all curves in $\sigma$.
When $\sigma$ consists of a single curve $\alpha$, we denote $S_\sigma$ by $S_\alpha$.
Each component of $S_\sigma$ is often identified with a complementary component of a tubular neighborhood of a one-dimensional submanifold representing $\sigma$ in $S$.
Let $W(S)$ denote the set of isotopy classes of the whole surface $S$ and such subsurfaces of $S$ as obtained from some $\sigma \in \Sigma(S)$.
For any $Q\in W(S)$, the set $V(Q)$ is naturally identified with a subset of $V(S)$.

We mean by a {\it pair of pants} a surface homeomorphic to $S_{0, 3}$.
We mean by a {\it handle} a surface homeomorphic to $S_{1, 1}$

\medskip

\noindent {\bf Complex $\calc(S)$.} When $3g+k-4>0$, we define $\calc(S)$ as the abstract simplicial complex so that the sets of vertices and simplices of $\calc(S)$ are $V(S)$ and $\Sigma(S)$, respectively.

When $S=S_{1, 1}$, we define $\calc(S)$ as the simplicial graph so that the set of vertices is $V(S)$ and two vertices $\alpha, \beta \in V(S)$ are adjacent if and only if $I(\alpha, \beta)=1$.
When $S=S_{0, 4}$, we define a simplicial graph $\calc(S)$ in the same manner after replacing the last condition by $I(\alpha, \beta)=2$.

The complex $\calc(S)$ is called the \textit{complex of curves} of $S$.

\medskip

The 1-skeleton of $\calc(S)$ is known to be a Gromov-hyperbolic metric space with respect to the graph distance (\cite{Mi96, MM99}).
Let $\partial \calc(S)$ denote the Gromov boundary of the 1-skeleton of $\calc(S)$.
We refer to \cite{Ha06, Kl99} for a description of $\partial \calc(S)$.


\subsection{Mapping class groups}

Let $S$ be a surface.
The {\it extended mapping class group} of $S$, denoted by $\mod^*(S)$, is the group of isotopy classes of homeomorphisms from $S$ onto itself, where isotopy may move points of the boundary of $S$.
The \textit{mapping class group} of $S$, denoted by $\mod(S)$, is the subgroup of $\mod^*(S)$ that consists of isotopy classes of orientation-preserving homeomorphisms from $S$ onto itself.
The {\it pure mapping class group} of $S$, denoted by $\pmod(S)$, is the subgroup of $\mod(S)$ that consists of isotopy classes of orientation-preserving homeomorphisms from $S$ onto itself which preserve each component of $\partial S$.
We refer to \cite{FM11, FLP79, Iv92, Iv02} for fundamentals of these groups.
We have the natural actions of $\mod^*(S)$ on the sets $V(S)$, $\Sigma(S)$, $W(S)$, $\calc(S)$, $\partial \calc(S)$, etc.
Let $\mod(S; 3)$ denote the subgroup of $\mod(S)$ consisting of elements acting on $H_1(S, \mathbb{Z}/3\mathbb{Z})$ trivially.
The group $\mod(S; 3)$ is torsion-free (\cite[Corollary 1.5]{Iv92}).

\begin{theorem}\label{thm-pure-t}\cite[Theorem 1.2 and Lemma 1.6]{Iv92}
Let $S=S_{g, k}$ be a surface with $3g+k-4\geq 0$.
Pick $h\in \mod(S; 3)$ and $\sigma \in \Sigma(S)$ with $h\sigma =\sigma$.
Then, for any $\alpha \in \sigma$, we have $h\alpha =\alpha$; and for any component $Q$ of $S_\sigma$, the element $h$ preserves $Q$ and each component of $\partial Q$, and the element of $\pmod(Q)$ induced by $h$ is either neutral or of infinite order.
\end{theorem}

Fix $\sigma \in \Sigma(S)$.
We set $\Gamma =\mod(S; 3)$ and denote by $\Gamma_\sigma$ the stabilizer of $\sigma$ in $\Gamma$.
We have the natural homomorphism
\[\theta_\sigma \colon \Gamma_\sigma \to \prod_Q\pmod(Q),\]
where $Q$ runs through all components of $S_\sigma$.
Let $D_\sigma$ denote the subgroup of $\pmod(S)$ generated by Dehn twists about curves in $\sigma$.
When $\sigma$ consists of a single curve $\alpha$, we denote $D_\sigma$ by $D_\alpha$.
The kernel of $\theta_\sigma$ is equal to $D_\sigma \cap \Gamma_\sigma$ (\cite[Corollary 4.1.B]{Iv02}).
For $Q\in W(S)$, we denote by $\Gamma_Q$ the stabilizer of $Q$ in $\Gamma$ and have the homomorphism $\theta_Q\colon \Gamma_Q\to \pmod(Q)$.

For $\alpha, \beta \in V(S)$, we say that $\alpha$ and $\beta$ \textit{fill} $S$ if there exists no $\gamma \in V(S)$ with $I(\alpha, \gamma)=0$ and $I(\beta, \gamma)=0$.
When $3g+k-4>0$, $\alpha$ and $\beta$ fill $S$ if and only if the graph distance between them in the 1-skeleton of $\calc(S)$ is at least 3.
When $S=S_{1, 1}$ or $S_{0, 4}$, any two distinct elements of $V(S)$ fill $S$.

\begin{lem}\cite[Proposition 2.8]{FM11}\label{lem-curve-stab}
Let $S=S_{g, k}$ be a surface with $3g+k-4\geq 0$.
Pick $\alpha, \beta \in V(S)$.
If $\alpha$ and $\beta$ fill $S$, then the stabilizer of $\alpha$ and $\beta$ in $\mod(S)$,
\[\{ \, h\in \mod(S)\mid h\alpha =\alpha,\ h\beta =\beta \, \},\]
is finite.
The stabilizer of $\alpha$ and $\beta$ in $\mod(S; 3)$ is therefore trivial.
\end{lem}


\subsection{Classification of subgroups}

Let $S=S_{g, k}$ be a surface with $3g+k-4\geq 0$.
An element of $\mod(S)$ is called \textit{reducible} if it fixes some element of $\Sigma(S)$.
By the celebrated Nielsen-Thurston classification, any element of $\mod(S)$ is fallen into exactly one of the following three kinds: elements of finite order; reducible elements of infinite order; and pseudo-Anosov elements.
Two pseudo-Anosov elements $h_1$, $h_2$ of $\mod(S)$ are called \textit{independent} if the fixed point sets of $h_1$ and $h_2$ in the Thurston boundary are disjoint.
In this case, sufficiently large powers of $h_1$ and $h_2$ generate a free group of rank 2 (\cite[Corollary 8.4]{Iv92}).

A subgroup of $\mod(S)$ is called \textit{reducible} if it fixes some element of $\Sigma(S)$.
Otherwise it is called \textit{irreducible}.
It is known that an infinite subgroup of $\mod(S)$ is irreducible if and only if it contains a pseudo-Anosov element (\cite[Corollary 7.14]{Iv92}).
A subgroup of $\mod(S)$ is called \textit{non-elementary} if it contains two independent pseudo-Anosov elements.
Any subgroup of $\mod(S)$ is known to be fallen into exactly one of the following four kinds: finite subgroups; infinite and reducible subgroups; infinite, irreducible and virtually cyclic subgroups; and non-elementary subgroups (\cite{MP89}, \cite[Corollary 7.15]{Iv92}).
The following theorem is obtained from this classification of subgroups, \cite[Proposition 5.1 (3)]{MP89} and \cite[Corollary 7.13]{Iv92}.

\begin{theorem}\label{thm-subgr}
Let $N$ and $G$ be infinite subgroups of $\mod(S)$ with $N\vartriangleleft G$.
Then, $N$ and $G$ are subgroups of the same kind.
\end{theorem}

Let $\Gamma$ be a non-trivial and reducible subgroup of $\mod(S; 3)$.
We then have a unique element $\sigma$ of $\Sigma(S)$ satisfying the following:
Any element of $\sigma$ is fixed by $\Gamma$, and for any $\alpha \in \sigma$ and any $\beta \in V(S)$ with $I(\alpha, \beta)\neq 0$, there exists an $h\in \Gamma$ with $h\beta \neq \beta$ (\cite[Corollary 7.12]{Iv92}).
This $\sigma$ is called the \textit{canonical reduction system (CRS)} of $\Gamma$.
The CRS $\sigma$ of $\Gamma$ is the minimal element of $\Sigma(S)$ satisfying the following:
Any element of $\sigma$ is fixed by $\Gamma$, and for any component $Q$ of $S_\sigma$, the subgroup $\theta_Q(\Gamma)$ of $\pmod(Q)$ is either trivial or infinite and irreducible (\cite[Theorem 7.16]{Iv92}).
Let $Q$ be a component of $S_\sigma$.
We say that $Q$ is \textit{T, IA}, or \textit{IN} for $\Gamma$ if $\theta_Q(\Gamma)$ is trivial; infinite, irreducible and virtually cyclic; or non-elementary as a subgroup of $\pmod(Q)$, respectively.
Here, ``T", ``IA" and ``IN" stands for ``trivial", ``irreducible and amenable" and ``irreducible and non-amenable", respectively.

\begin{lem}\label{lem-icc}
Let $S=S_{g, k}$ be a surface with $3g+k-4>0$ and $(g, k)\neq (1, 2), (2, 0)$.
Let $\Gamma$ be an infinite, normal subgroup of some finite index subgroup of $\mod(S)$.
Then, the Dirac measure on the neutral element is the only probability measure on $\mod^*(S)$ invariant under conjugation by $\Gamma$.
\end{lem}

\begin{proof}
Let $\mathcal{T}$ and $\mathcal{PMF}$ denote the Teichm\"uller space and the Thurston boundary for $S$, respectively, on which $\mod^*(S)$ naturally acts.
As precisely proved in the next paragraph, we first note that the limit set of $\Gamma$ in $\mathcal{PMF}$ is the whole space $\mathcal{PMF}$.

Let $G$ be a non-elementary subgroup of $\mod(S)$.
Denote by $\Lambda_0(G)\subset \mathcal{PMF}$ the set of fixed points of pseudo-Anosov elements in $G$.
Following \cite{MP89}, we call the closure of $\Lambda_0(G)$ in $\mathcal{PMF}$ the \textit{limit set} of $G$ in $\mathcal{PMF}$.
As noted in \cite[Section 5, Example 1]{MP89}, by \cite[Expos\'e 6, \S VII, Th\'eor\`eme]{FLP79}, the limit set of $\mod(S)$ is $\mathcal{PMF}$.
This is also proved in \cite[Lemma 9.12]{Iv92}.
It follows from \cite[Propositions 5.4 and 5.5]{MP89} that for any infinite, normal subgroup of a finite index subgroup of $\mod(S)$, its limit set is $\mathcal{PMF}$.
The set $\Lambda_0(\Gamma)$ is thus dense in $\mathcal{PMF}$.

The rest of the proof follows the proof of \cite[Theorem 2.9]{Ki06}.
For the reader's convenience, we give a complete proof.
Pick $\gamma_0\in \mod^*(S)$ such that the set $\{ \, \gamma \gamma_0\gamma^{-1}\mid \gamma \in \Gamma \, \}$ is finite.
We set $\textrm{Fix}(\gamma_0)=\{ \, x\in \mathcal{T}\cup \mathcal{PMF}\mid \gamma_0x =x\, \}$.
If the inclusion $\mathcal{PMF}\subset \textrm{Fix}(\gamma_0)$ is shown, then it implies $\gamma_0=e$, and the lemma follows.
Assuming $\mathcal{PMF}\not\subset \textrm{Fix}(\gamma_0)$, we deduce a contradiction.
For a pseudo-Anosov element $\gamma$ of $\mod(S)$, let $F_\pm(\gamma)$ denote its two fixed points in $\mathcal{PMF}$ so that for any $x\in \mathcal{T}\cup \mathcal{PMF}$ except for $F_-(\gamma)$, $\gamma^nx$ converges to $F_+(\gamma)$ as $n\to +\infty$, and $F_-(\gamma)=F_+(\gamma^{-1})$. 
By density of $\Lambda_0(\Gamma)$ in $\mathcal{PMF}$, there exist two independent pseudo-Anosov element $\gamma_1, \gamma_2\in \Gamma$ with $F_+(\gamma_1), F_+(\gamma_2)\not\in \textrm{Fix}(\gamma_0)$.

As the union $\mathcal{T}\cup \mathcal{PMF}$ is homeomorphic to a finite-dimensional Euclidean closed ball by \cite[Expos\'e 11, \S II, Th\'eor\`eme 3]{FLP79}, the Brouwer fixed point theorem implies that $\textrm{Fix}(\gamma_0)$ is non-empty.
Pick $x\in \textrm{Fix}(\gamma_0)$.
There exists a subsequence $\{ n_i\}_i$ of $\mathbb{N}$ with $\gamma_1^{-n_i}\gamma_0\gamma_1^{n_i}=\gamma_0$ for any $i$.
The equation $\gamma_1^{-n_i}\gamma_0\gamma_1^{n_i}x=\gamma_0x=x$ then holds, and therefore $\gamma_1^{n_i}x\in \textrm{Fix}(\gamma_0)$ for any $i$.
If $x$ were not $F_-(\gamma_1)$, then $\gamma_1^{n_i}x$ would converge to $F_+(\gamma_1)\not\in \textrm{Fix}(\gamma_0)$ as $i\to \infty$.
This is a contradiction.
We thus have $x=F_-(\gamma_1)$.
Similarly, we have $x=F_-(\gamma_2)$.
This contradicts that $\gamma_1$ and $\gamma_2$ are independent.
\end{proof}

Throughout the following lemma and its proof, for any subgroup $\Gamma$ of $\pmod(S)$ and any $\sigma \in \Sigma(S)$, we will denote by $\Gamma_\sigma$ the stabilizer of $\sigma$ in $\Gamma$.

\begin{lem}\label{lem-group-chain}
Let $\Gamma$ be a subgroup of $\mod(S; 3)$ with $\Gamma \vartriangleleft \pmod(S)$.
Fix $\sigma \in \Sigma(S)$ and assume that there exist two distinct components $Q$, $R$ of $S_\sigma$ such that $\theta_Q(\Gamma_\sigma)$ and $\theta_R(\Gamma_\sigma)$ are non-trivial. Then, there exist subgroups $M_1$, $M_2$ and $N_1$ of $\Gamma$ such that $N_1$ is infinite and amenable, $M_2$ is non-amenable, and we have $N_1\vartriangleleft M_1$, $M_2\vartriangleleft M_1$ and $M_2\vartriangleleft \Gamma_\sigma$.
\end{lem}

\begin{proof}
For $X\in W(S)$, let us say that an element of $\pmod(S)$ is \textit{supported} on $X$ if it has a representative which is the identity outside $X$.
By assumption, we have $\gamma, \delta \in \Gamma_\sigma$ such that $\theta_Q(\gamma)$ and $\theta_R(\delta)$ are non-neutral.
Pick an element $\gamma_1$ of $\pmod(S)_\sigma$ such that it is supported on $Q$ and the element $\theta_Q([\gamma_1, \gamma])$ is non-neutral, where for two elements $x$, $y$ of a group, $[x, y]$ denotes the commutator $xyx^{-1}y^{-1}$.
Such a $\gamma_1$ is found as follows:
There exists $\gamma_1'\in \pmod(Q)$ with $[\gamma_1', \theta_Q(\gamma)]$ non-neutral because the center of $\pmod(Q)$ is finite (\cite[Section 3.4]{FM11}) and $\theta_Q(\gamma)$ is not central by Theorem \ref{thm-pure-t}.
Let $\gamma_1$ be an element of $\pmod(S)$ supported on $Q$ with $\theta_Q(\gamma_1)=\gamma_1'$.
This is a desired one.
Similarly we can find an element $\delta_1$ of $\pmod(S)_\sigma$ such that it is supported on $R$ and the element $\theta_R([\delta_1, \delta])$ is non-neutral.
Since $\Gamma \vartriangleleft \pmod(S)$, the elements $[\gamma_1, \gamma]$ and $[\delta_1, \delta]$ belong to $\Gamma$, and are moreover supported on $Q$ and $R$, respectively.

We set $L_R=(\bigcap_X\ker \theta_X)\cap \Gamma_\sigma$, where $X$ runs through all components of $S_\sigma$ other than $R$.
We have $\theta_R(L_R)\vartriangleleft \theta_R(\Gamma_\sigma)$ and $\theta_R(\Gamma_\sigma)\vartriangleleft \theta_R(\mod(S; 3)_\sigma)$ because $\Gamma \vartriangleleft \pmod(S)$.
Since $[\delta_1, \delta]$ belongs to $L_R$, the group $\theta_R(L_R)$ is non-trivial.
By Theorems \ref{thm-pure-t} and \ref{thm-subgr}, $\theta_R(L_R)$ is a non-elementary subgroup of $\pmod(R)$, and $L_R$ is therefore non-amenable.

Let $N_1$ be the group generated by $[\gamma_1, \gamma]$.
The group $N_1$ is infinite because $[\gamma_1, \gamma]$ is a non-neutral element of $\mod(S; 3)$.
We set $M_2=\ker \theta_Q\cap \Gamma_\sigma$ and $M_1=N_1\vee M_2$.
The groups $N_1$ and $M_2$ commute, and we thus have $N_1\vartriangleleft M_1$ and $M_2\vartriangleleft M_1$.
We also have $M_2\vartriangleleft \Gamma_\sigma$.
The group $M_2$ contains $L_R$ and is therefore non-amenable.
\end{proof}



\section{Groupoids from mapping class groups}\label{sec-groupoids-mcg}

\subsection{Discrete measured groupoids}

We refer to \cite{ADR00} and \cite[Chapter XIII]{Ta03} for discrete measured groupoids and their amenability.
By a \textit{standard probability space} we mean a standard Borel space equipped with a probability measure. All relations involving measurable sets and maps that appear throughout the paper are understood to hold up to sets of measure zero, unless otherwise mentioned.
Let $(X, \mu)$ be a standard probability space.
A \textit{non-negligible} subset of $A\subset X$ is a measurable subset satisfying $\mu(A)>0$.
Let $\calg$ be a discrete measured groupoid on $(X, \mu)$ with the range and source maps $r, s\colon \calg \to X$.
A subgroupoid of $\calg$ is a measurable one of $\calg$ whose unit space is $(X, \mu)$.
For a non-negligible subset $A\subset X$, we denote by
\[\calg|_A=\{ \, g\in \calg\,:\, r(g), s(g)\in A\, \}\]
the restriction of $\calg$ to $A$.

Let $\Gamma$ be a countable group, $M$ a standard Borel space on which $\Gamma$ acts, and $\rho \colon \calg \to \Gamma$ a homomorphism.
Let $A\subset X$ be a measurable subset.
A measurable map $\varphi \colon A\to M$ is called \textit{$(\calg, \rho)$-invariant} if $\rho(g)\varphi(s(g))=\varphi(r(g))$ for almost every $g\in \calg$.
Suppose that $K$ is a compact metrizable space on which $\Gamma$ acts continuously and that $M$ is the space of probability measures on $K$.
If $\calg$ is amenable, then there exists a $(\calg, \rho)$-invariant map from $X$ into $M$.

We say that $\calg$ is \textit{finite} if for almost every $x\in X$, the set $r^{-1}(x)$ is finite.
We say that $\calg$ is \textit{nowhere finite} if for almost every $x\in X$, the set $r^{-1}(x)$ is infinite.
We say that $\calg$ is \textit{nowhere amenable} if for any non-negligible subset $A$ of $X$, the restriction $\calg|_A$ is not amenable.

Let $\Gamma$ be a countable group and $\Gamma \ca (X, \mu)$ a p.m.p.\ action.
The product space $\Gamma \times X$ has the following natural structure of a groupoid on $X$:
the range and source maps are defined by $r(\gamma, x)=\gamma x$ and $s(\gamma, x)=x$, respectively, for $\gamma \in \Gamma$ and $x\in X$.
The product is defined by $(\gamma_1, \gamma_2x)(\gamma_2, x)=(\gamma_1\gamma_2, x)$ for $\gamma_1, \gamma_2\in \Gamma$ and $x\in X$.  
The element $(e, x)$ is the unit at $x\in X$. 
The inverse of $(\gamma, x)\in \Gamma \times X$ is defined by $(\gamma^{-1}, \gamma x)$.
This groupoid also has the structure of a discrete measured groupoid on $(X, \mu)$ and is denoted by $\Gamma \ltimes (X, \mu)$;  when $\mu$ is self-understood from the context, then it will be denoted by $\Gamma \ltimes X$.

When $\cals$ is a normal subgroupoid of $\calg$, we write $\cals \vartriangleleft \calg$.
We refer to \cite[Definition 2.2]{Ki06} for a definition of normal subgroupoids.
The following lemma collects basic facts on normal subgroupoids:

\begin{lem}\label{lem-normal}\cite[Lemmas 2.13 and 2.16]{Ki06}
The following assertions hold:
\begin{enumerate}
\item[(i)] Let $\Gamma$ be a countable group, $N$ a normal subgroup of $\Gamma$ and $\Gamma \ca (X, \mu)$ a p.m.p.\ action.
Then $N\ltimes X \vartriangleleft \Gamma \ltimes X$.
\item[(ii)] Let $\calg$ be a discrete measured groupoid on $(X, \mu)$, $\caln$ a subgroupoid of $\calg$ with $\caln \vartriangleleft \calg$, and $A$ a non-negligible subset of $X$.
Then $\caln |_A\vartriangleleft \calg |_A$.
\end{enumerate}
\end{lem}

\subsection{Preliminary results on subgroupoids}\label{subsec-pre-res}

First we summarize some results from \cite{Ki05, Ki06} regarding groupoids associated with p.m.p.\ actions of mapping class groups. Afterwards we show a few consequences which will be used in the sequel.

Throughout this subsection, we fix the following notation:
Let $S=S_{g, k}$ be a surface with $3g+k-4>0$.
Let $\Gamma$ be an infinite subgroup of $\mod(S; 3)$ and $\Gamma \ca (X, \mu)$ a p.m.p.\ action.
We set $\calg =\Gamma \ltimes X$ and define $\rho \colon \calg \to \Gamma$ as the projection.
For a set $V$ on which $\Gamma$ acts (e.g., $V(S)$, $\Sigma(S)$ and $W(S)$) and for $v\in V$, we denote by $\Gamma_v$ the stabilizer of $v$ in $\Gamma$ and set $\calg_v =\Gamma_v \ltimes X$.
For $\sigma \in \Sigma(S)$ and $Q\in W(S)$, we define homomorphisms $\rho_\sigma$, $\rho_Q$ by
\[\rho_\sigma =\theta_\sigma \circ \rho \colon \calg_\sigma \to \prod_R\pmod(R)\quad \textrm{and}\quad \rho_Q=\theta_Q\circ \rho \colon \calg_Q\to \pmod(Q),\]
where $R$ runs through all components of $S_\sigma$.
Let $Y$ be a non-negligible subset of $X$ and $\cals$ a nowhere finite subgroupoid of $\calg|_Y$.

\begin{definition}
With the above notation,
\begin{enumerate}
\item[(i)] we say that $\cals$ is \textit{reducible} if there exists an $(\cals, \rho)$-invariant map from $Y$ into $\Sigma(S)$. 
\item[(ii)] For $\alpha \in V(S)$ and a measurable subset $A$ of $Y$, we say that the pair $(\alpha, A)$ is $(\cals, \rho)$-\textit{invariant} if there exists a partition $A=\bigsqcup_nA_n$ into countably many measurable subsets such that for any $n$, the constant map from $A_n$ into $V(S)$ whose value is $\alpha$ is $(\cals, \rho)$-invariant.
\item[(iii)] Suppose that $(\alpha, A)$ is a $(\cals, \rho)$-invariant pair.
We say that $(\alpha, A)$ is \textit{purely $(\cals, \rho)$-invariant} if for any $\beta \in V(S)$ with $I(\alpha, \beta)\neq 0$ and for any non-negligible subset $B$ of $A$, the pair $(\beta, B)$ is not $(\cals, \rho)$-invariant.
\end{enumerate}
\end{definition}

\begin{Remark}
In \cite{Ki06}, reducible subgroupoids are defined in a different way.
We show that the two definitions are equivalent.
Let $\mathcal{PMF}$ denote the Thurston boundary for $S$ and $\mathcal{MIN}$ denote the subset of $\mathcal{PMF}$ consisting of minimal measured foliations.
The group $\mod(S)$ naturally acts on $\mathcal{PMF}$, and $\mathcal{MIN}$ is a subset invariant under this action. 
Let $M(\mathcal{PMF})$ be the space of probability measures on $\mathcal{PMF}$, on which $\mod(S)$ naturally acts.
In \cite[Definition 3.2]{Ki06}, a subgroupoid $\cals$ of $\calg |_Y$ is called reducible if there exists an $(\cals, \rho)$-invariant map $\psi \colon Y\to M(\mathcal{PMF})$ such that for almost every $x\in Y$, the measure $\psi(x)$ is supported on the complement of $\mathcal{MIN}$.
If such a map $\psi$ exists, then there exists an $(\cals, \rho)$-invariant map from $Y$ into $\Sigma(S)$ (\cite[Theorem 3.6]{Ki06}).
The converse also holds because we have the embedding of $\Sigma(S)$ into the complement of $\mathcal{MIN}$ that is equivariant under the action of $\mod(S)$ (\cite[Remarque in Expos\'e 4, \S II]{FLP79}).
\end{Remark}

\begin{theorem}\cite[Theorem 3.6]{Ki06}
Suppose that $\cals$ is reducible.
Then, there exists a unique $(\cals, \rho)$-invariant map $\varphi \colon Y\to \Sigma(S)$ satisfying the following conditions (1) and (2):
\begin{enumerate}
\item For any $\sigma \in \Sigma(S)$ with $\varphi^{-1}(\sigma)$ non-negligible and for any $\alpha \in \sigma$, the pair $(\alpha, \varphi^{-1}(\sigma))$ is purely $(\cals, \rho)$-invariant.
\item If $(\alpha, A)$ is a purely $(\cals, \rho)$-invariant pair, then $\alpha \in \varphi(x)$ for almost every $x\in A$.
\end{enumerate}
\end{theorem}

Following terminology for reducible subgroups, we call this $(\cals, \rho)$-invariant map $\varphi$ the \textit{canonical reduction system (CRS)} of $\cals$.
For any non-negligible subset $A$ of $Y$, the CRS of $\cals|_A$ is the restriction of $\varphi$ to $A$.

\begin{lem}\label{lem-nor-red}
Let $\caln$ be a nowhere finite subgroupoid of $\cals$ with $\caln \vartriangleleft \cals$.
Then, the following assertions hold:
\begin{enumerate}
\item[(i)] If $\caln$ is amenable and $\cals$ is nowhere amenable, then $\cals$ is reducible.
\item[(ii)] If $\caln$ is reducible, then the CRS of $\caln$ is $(\cals, \rho)$-invariant.
In particular, $\cals$ is reducible.
\item[(iii)] In assertion (ii), let $\varphi, \psi \colon Y\to \Sigma(S)$ denote the CRS's of $\caln$ and $\cals$, respectively.
Then $\varphi(x)\subset \psi(x)$ for almost every $x\in Y$.
\end{enumerate}
\end{lem}

\begin{proof}
Assertion (ii) is \cite[Theorem 3.6 (iii)]{Ki06}.
Assertion (i) follows from assertion (ii) and \cite[Proposition 4.1]{Ki06}.
Assertion (iii) holds because if a pair $(\alpha, A)$ is purely $(\caln, \rho)$-invariant and is $(\cals, \rho)$-invariant, then it is purely $(\cals, \rho)$-invariant by definition.
\end{proof}

The set $W(S)$ was introduced in Subsection \ref{subsec-curve}.
We define $\Omega(S)$ as the set of finite subsets $F$ of $W(S)$ (including the empty set) such that any two distinct elements of $F$ have disjoint representatives in $S$.
For $\sigma \in \Sigma(S)$, let $\omega_\sigma \in \Omega(S)$ denote the set of components of $S_\sigma$.
For a surface $Q=S_{g', k'}$ with $3g'+k'-4\geq 0$, we define $\partial_2\calc(Q)$ as the quotient space of $\partial \calc(Q)\times \partial \calc(Q)$ obtained by identifying any its point $(x, y)$ with $(y, x)$.

\begin{theorem}\label{thm-tiain}\cite[Theorems 3.13 and 3.15]{Ki06}
Suppose that $\cals$ is reducible and denote by $\varphi \colon Y\to \Sigma(S)$ the CRS of $\cals$.
Then, there exist unique $(\cals, \rho)$-invariant maps $\varphi_t, \varphi_{ia}, \varphi_{in}\colon Y\to \Omega(S)$ satisfying the following conditions (1)--(4):
\begin{enumerate}
\item[(1)] For almost every $x\in Y$, the equation $\varphi_t(x)\cup \varphi_{ia}(x)\cup \varphi_{in}(x)=\omega_{\varphi(x)}$ holds, the sets $\varphi_t(x)$, $\varphi_{ia}(x)$ and $\varphi_{in}(x)$ are mutually disjoint, and any element of $\omega_{\varphi(x)}$ that is a pair of pants belongs to $\varphi_t(x)$.
\item[(2)] Pick $F\in \Omega(S)$ with $\varphi_t^{-1}(F)$ non-negligible and pick $Q\in F$.
Then, for any $\alpha \in V(Q)$, the pair $(\alpha, \varphi_t^{-1}(F))$ is $(\cals, \rho)$-invariant.
\item[(3)] Pick $F\in \Omega(S)$ with $\varphi_{ia}^{-1}(F)$ non-negligible and pick $Q\in F$.
Then, for any non-negligible subset $A$ of $Y$, there exists no $(\cals, \rho_Q)$-invariant map from $A$ into $\Sigma(Q)$, and there exists an $(\cals, \rho_Q)$-invariant map from $Y$ into $\partial_2\calc(Q)$.
\item[(4)] Pick $F\in \Omega(S)$ with $\varphi_{in}^{-1}(F)$ non-negligible and pick $Q\in F$.
Then, for any non-negligible subset $A$ of $Y$, there exists no $(\cals, \rho_Q)$-invariant map from $A$ into $\Sigma(Q)\cup \partial_2\calc(Q)$.
\end{enumerate}
\end{theorem}

Following terminology for reducible subgroups again, we call the maps $\varphi_t$, $\varphi_{ia}$ and $\varphi_{in}$ in Theorem \ref{thm-tiain} the \textit{T, IA} and \textit{IN systems} of $\cals$, respectively.
For any non-negligible subset $A$ of $Y$, the T, IA and IN systems of $\cals|_A$ are the restrictions of $\varphi_t$, $\varphi_{ia}$ and $\varphi_{in}$ to $A$, respectively.
We often reduce our argument to the case where the CRS of $\cals$ and all of these three maps are constant on $Y$.
In this case, we call elements of the constant values of $\varphi_t$, $\varphi_{ia}$ and $\varphi_{in}$ the \textit{T, IA} and \textit{IN components} of $\cals$, respectively.

\begin{Remark}\label{rem-t-comp}
Suppose that $\cals$ is reducible and that all of the CRS and the T, IA and IN systems of $\cals$ are constant.
Let $Q$ be a T component of $\cals$.
By Lemma \ref{lem-curve-stab} and condition (2) in Theorem \ref{thm-tiain}, there exists a partition $Y=\bigsqcup_nY_n$ into countably many measurable subsets such that for any $n$, the image $\rho_Q(\cals|_{Y_n})$ consists of only the neutral element.
\end{Remark}

\begin{theorem}\label{thm-ia-max}\cite[Theorem 3.15 and Lemma 3.16]{Ki06}
Suppose that $\cals$ is reducible and all of the CRS and the T, IA and IN systems of $\cals$ are constant.
Let $Q$ be an IA component of $\cals$.
Then, the following assertions hold:
\begin{enumerate}
\item[(i)] There exists a unique $(\cals, \rho_Q)$-invariant map $\psi_0\colon Y\to \partial_2\calc(Q)$ satisfying the following maximality:
For any non-negligible subset $A$ of $Y$ and any $(\cals, \rho_Q)$-invariant map $\psi \colon A\to \partial_2\calc(Q)$, the inclusion $\psi(x)\subset \psi_0(x)$ holds for almost every $x\in A$, where each point of $\partial_2\calc(Q)$ is naturally regarded as a non-empty subset of $\partial \calc(Q)$ consisting of at most two points.
\item[(ii)] Let $\calm$ be a subgroupoid of $\calg|_Y$ with $\cals \vartriangleleft \calm$.
Note that $\calm$ is reducible by Lemma \ref{lem-nor-red} (ii).
Suppose that all of the CRS and the T, IA and IN systems of $\calm$ are constant.
Then, the map $\psi_0$ in assertion (i) is $(\calm, \rho_Q)$-invariant, and $Q$ is an IA component of $\calm$.
\end{enumerate}
\end{theorem}

Based on these general results from \cite{Ki06} on reducible subgroupoids, in the rest of this section, we prove several lemmas for later use.

\begin{lem}\label{lem-no-in}
Suppose that $\cals$ is reducible and amenable.
Let $\varphi_{in}\colon Y\to \Omega(S)$ be the IN system of $\cals$.
Then $\varphi_{in}(x)=\emptyset$ for almost every $x\in Y$.
\end{lem}

\begin{proof}
We may assume that all of the CRS and the T, IA and IN systems of $\cals$ are constant.
Let $\sigma \in \Sigma(S)$ be the value of the CRS of $\cals$.
Let $Q$ be a component of $S_\sigma$ and $\mathcal{PMF}$ denote the Thurston boundary for $Q$.
As stated in \cite[Theorem 3.15 (iii)]{Ki06}, if $Q$ were IN for $\cals$, then there would exist no $(\cals, \rho_Q)$-invariant map from $Y$ into the space of probability measures on $\mathcal{PMF}$.
By amenability of $\cals$, there however exists such a map because $\mathcal{PMF}$ is compact.
It follows that $Q$ is not an IN component of $\cals$.
\end{proof}

The following lemma is an immediate consequence of \cite[Proposition 3.17]{Ki06}:

\begin{lem}\label{lem-ame}
Suppose that $\cals$ is reducible and nowhere amenable.
Let $\varphi_{in}\colon Y\to \Omega(S)$ be the IN system of $\cals$.
Then $\varphi_{in}(x)\neq \emptyset$ for almost every $x\in Y$.
\end{lem}

\begin{lem}\label{lem-red-group}
Suppose that $\Gamma$ is reducible and denote by $\sigma \in \Sigma(S)$ the CRS of $\Gamma$.
Then, all of the CRS and the T, IA and IN systems of $\calg =\Gamma \ltimes X$ are constant, and their values are equal to those of $\Gamma$, respectively.
\end{lem}

\begin{proof}
We suppose that all of the CRS and the T, IA and IN systems of $\calg |_Y$ are constant, and show the conclusion of the lemma holds for $\calg|_Y$.
This is enough for the lemma.
The assertion on the CRS is \cite[Lemma 3.8]{Ki06}.
Let $Q$ be a component of $S_\sigma$.
If $Q$ is T for $\Gamma$, then either $Q$ is a pair of pants or any element of $V(Q)$ is fixed by any element of $\theta_Q(\Gamma)$.
In the latter case, for any $\alpha \in V(Q)$, the pair $(\alpha, Y)$ is $(\calg, \rho_Q)$-invariant.
The component $Q$ is therefore T for $\calg|_Y$.

Suppose that $Q$ is IA for $\Gamma$.
There exists a $g\in \Gamma$ such that $\theta_Q(g)$ is pseudo-Anosov and generates a finite index subgroup of $\theta_Q(\Gamma)$.
Since any non-zero power of $\theta_Q(g)$ fixes no element of $V(Q)$, the component $Q$ is not T for $\calg|_Y$.
Since any non-zero power of $\theta_Q(g)$ fixes exactly two points of $\partial \calc(Q)$, the component $Q$ is IA for $\calg|_Y$.

Suppose that $Q$ is IN for $\Gamma$.
There exist $g_1, g_2\in \Gamma$ such that $\theta_Q(g_1)$ and $\theta_Q(g_2)$ are independent pseudo-Anosov elements.
Let $\{ F_1^\pm \}$ and $\{ F_2^\pm \}$ denote the fixed point sets of $\theta_Q(g_1)$ and $\theta_Q(g_2)$ in $\partial \calc(Q)$, respectively.
As in the previous paragraph, $Q$ is not T for $\calg|_Y$.
If $Q$ were IA for $\calg|_Y$, then there would exist a $(\calg, \rho_Q)$-invariant map $\psi \colon Y\to \partial_2\calc(Q)$.
For $i=1, 2$, we set $\Gamma_i=\langle g_i\rangle$ and set $\calg_i=(\Gamma_i\ltimes X)|_Y$.
By the assertion proved in the previous paragraph, $Q$ is IA for both $\calg_1$ and $\calg_2$.
For $i=1, 2$, let $\psi_i \colon Y\to \partial_2\calc(Q)$ be the $(\calg_i, \rho_Q)$-invariant map satisfying the maximality in Theorem \ref{thm-ia-max} (i).
This map is constant, and its value is $\{ F_i^\pm\}$ because that constant map is $(\calg_i, \rho_Q)$-invariant.
On the other hand, $\psi$ is $(\calg_i, \rho_Q)$-invariant, and the maximality of $\psi_i$ implies that the inclusion $\psi(x)\subset \psi_i(x)$ holds for almost every $x\in Y$.
This contradicts independence between $\theta_Q(g_1)$ and $\theta_Q(g_2)$.
We have shown that $Q$ is not IA for $\calg|_Y$, and $Q$ is therefore IN for $\calg|_Y$.
\end{proof}

\begin{lem}\label{lem-nor-in}
Let $\caln$ be a nowhere finite subgroupoid of $\cals$ with $\caln \vartriangleleft \cals$.
Suppose that $\cals$ is reducible and that all of the CRS's and the T, IA and IN systems of $\cals$ and $\caln$ are constant.
Let $Q$ be an IN component of $\cals$.
Then, either there exists a T component $R$ of $\caln$ with $Q\subset R$ or $Q$ is an IN component of $\caln$.
\end{lem}

\begin{proof}
Let $\sigma, \tau \in \Sigma(S)$ be the values of the CRS's of $\cals$ and $\caln$, respectively.
Since $\caln \vartriangleleft \cals$, we have $\tau \subset \sigma$ by Lemma \ref{lem-nor-red} (iii).
We thus have the component $R$ of $S_\tau$ with $Q\subset R$.
If $R$ is T for $\caln$, then the lemma follows.
Suppose that $R$ is not T for $\caln$.
For any $\alpha \in \sigma$, the pair $(\alpha, Y)$ is $(\caln, \rho)$-invariant, and thus $\alpha$ is not a curve in $R$.
The equation $Q=R$ follows.
If $Q$ were IA for $\caln$, then by Theorem \ref{thm-ia-max} (ii), the assumption $\caln \vartriangleleft \cals$ would imply that $Q$ is IA for $\cals$.
This is a contradiction.
It follows that $Q$ is IN for $\caln$.
\end{proof}

\begin{lem}\label{lem-in-comp}
Let $\alpha \in V(S)$ be a non-separating curve in $S$.
Suppose that $D_\alpha \cap \Gamma$ is trivial.
Let $\caln$ be a nowhere finite subgroupoid of $\cals$ with $\caln \vartriangleleft \cals$.
Suppose that $\cals$ is reducible and that all of the CRS's and the T, IA and IN systems of $\cals$ and $\caln$ are constant.
Then, the following conditions (1) and (2) are equivalent:
\begin{enumerate}
\item[(1)] The value of the CRS of $\cals$ is $\{ \alpha \}$, and the component $S_\alpha$ is IN for $\cals$.
\item[(2)] The value of the CRS of $\caln$ is $\{ \alpha \}$, and the component $S_\alpha$ is IN for $\caln$.
\end{enumerate}
\end{lem}

\begin{proof}
We first assume condition (1).
By Lemma \ref{lem-nor-red} (iii), the value of the CRS of $\caln$ is $\{ \alpha \}$.
By Lemma \ref{lem-nor-in}, the component $S_\alpha$ is either T for $\caln$ or IN for $\caln$.
If $S_\alpha$ were T for $\caln$, then there exists a non-negligible subset $A$ of $Y$ such that $\rho(\caln |_A)\subset \ker \theta_{S_\alpha}<D_\alpha$.
Since $D_\alpha \cap \Gamma$ is assumed to be trivial, $\caln |_A$ is trivial.
This contradicts that $\caln$ is nowhere finite.
We have shown that $S_\alpha$ is IN for $\caln$, and condition (2) follows.

We next assume condition (2).
Let $\sigma \in \Sigma(S)$ be the value of the CRS of $\cals$.
By Lemma \ref{lem-nor-red} (iii), $\sigma$ contains $\alpha$.
Pick $\beta \in \sigma$.
The pair $(\beta, Y)$ is $(\caln, \rho)$-invariant because $\caln <\cals$.
We have $\beta =\alpha$ because $S_\alpha$ is IN for $\caln$.
We therefore have $\sigma =\{ \alpha \}$.
The component $S_\alpha$ is IN for $\cals$ because it is IN for $\caln$.
Condition (1) follows.
\end{proof}

\begin{lem}\label{lem-s-chain}
Pick $\alpha \in V_s(S)$.
We suppose the following conditions (a)--(c):
\begin{enumerate}
\item[(a)] The CRS of $\Gamma_\alpha$ is $\{ \alpha \}$.
\item[(b)] One component of $S_\alpha$ is IN for $\Gamma_\alpha$, and another component of $S_\alpha$ is either T or IN for $\Gamma_\alpha$.
\item[(c)] If there is a T component of $\Gamma_\alpha$, denoted by $Q$, then for any $R\in W(Q)$, the group $\theta_R(\Gamma_R)$ is trivial.
\end{enumerate}
Let $\call$, $\calm_1$, $\calm_2$ and $\caln_1$ be subgroupoids of $\calg|_Y$ such that $\calm_2$ is nowhere amenable, $\caln_1$ is amenable and nowhere finite, and we have $\calg_\alpha |_Y < \call$, $\caln_1\vartriangleleft \calm_1$, $\calm_2\vartriangleleft \calm_1$ and $\calm_2 \vartriangleleft \call$.
Then $\calg_\alpha |_Y =\call$.
\end{lem}

\begin{proof}
By Lemma \ref{lem-nor-red} (i) and (ii), all of $\call$, $\calm_1$, $\calm_2$ and $\caln_1$ are reducible.
We will show that the CRS of $\call$ is constant and its value contains $\alpha$.
This implies $\call <\calg_\alpha |_Y$, and the lemma follows.
We may therefore assume that all of the CRS's and the T, IA and IN systems of $\call$, $\calm_1$, $\calm_2$ and $\caln_1$ are constant.
Let $\sigma, \sigma_1, \sigma_2, \tau_1\in \Sigma(S)$ denote the values of the CRS's of $\call$, $\calm_1$, $\calm_2$ and $\caln_1$, respectively.
We set $\calm =\calg_\alpha |_Y$.
The CRS and the T, IA and IN systems of $\calm$ are constant, and their values are the same as those of $\Gamma_\alpha$ by Lemma \ref{lem-red-group}.

We show that $\sigma$ contains $\alpha$.
Assuming to the contrary that $\sigma$ does not contain $\alpha$, we deduce a contradiction.
Pick $\beta \in \sigma$.
The pair $(\beta, Y)$ is purely $(\call, \rho)$-invariant, and is $(\calm, \rho)$-invariant because $\calm <\call$.
We have $I(\alpha, \beta)=0$ because $(\alpha, Y)$ is purely $(\calm, \rho)$-invariant.
The curve $\alpha$ is separating in $S$, and $S_\alpha$ therefore consists of exactly two components.
The curve $\beta$ does not lie in any IN component of $\Gamma_\alpha$ because $(\beta, Y)$ is $(\calm, \rho)$-invariant.
By condition (b), one component of $S_\alpha$ is IN for $\Gamma_\alpha$, another component of $S_\alpha$, denoted by $Q$, is T for $\Gamma_\alpha$, and we have $\beta \in V(Q)$.
Since $\beta$ is an arbitrary element of $\sigma$, we have $\sigma \in \Sigma(Q)$.

By condition (c), any component of $S_\sigma$ contained in $Q$ is T for $\call$.
Let $R$ denote the component of $S_\sigma$ containing the IN component of $\Gamma_\alpha$.
This $R$ is the unique IN component of $\call$, and we have $\alpha \in V(R)$.

By Lemma \ref{lem-nor-in}, the inclusion $\calm_2\vartriangleleft \call$ implies that either there exists a T component of $\calm_2$ containing $R$ or $R$ is an IN component of $\calm_2$.
If the former were true, then by condition (c), any other component of $S_{\sigma_2}$ would also be T for $\calm_2$.
By Lemma \ref{lem-ame}, this contradicts that $\calm_2$ is nowhere amenable.
It follows that $R$ is an IN component of $\calm_2$.

By Lemma \ref{lem-nor-red} (iii), the inclusion $\calm_2\vartriangleleft \calm_1$ implies $\sigma_2\subset \sigma_1$.
No curve in $\sigma_1$ lies in $R$ because $R$ is IN for $\calm_2$.
It follows that $\sigma_1\in \Sigma(Q)$ and that $R$ is a component of $S_{\sigma_1}$ and is IN for $\calm_1$ because it is IN for $\calm_2$.

The inclusion $\caln_1\vartriangleleft \calm_1$ implies $\tau_1\subset \sigma_1$ and $\tau_1\in \Sigma(Q)$.
By Lemma \ref{lem-nor-in}, either there exists a T component of $\caln_1$ containing $R$ or $R$ is an IN component of $\caln_1$.
The latter does not hold because $\caln_1$ is amenable.
By condition (c), the former implies that any component of $S_{\tau_1}$ is T for $\caln_1$.
By condition (c) again, $\theta_Q(\Gamma_\alpha)$ is trivial, and $D_{\tau_1}\cap \Gamma$ is therefore trivial.
There exists a non-negligible subset $A$ of $Y$ such that $\rho(\caln_1|_A)$ consists of only the neutral element.
This contradicts that $\caln_1$ is nowhere finite.
\end{proof}

The following lemma for a BP is an analogue of Lemma \ref{lem-s-chain}.

\begin{lem}\label{lem-bp-chain}
Pick $b\in V_{bp}(S)$.
We suppose the following conditions (a)--(c):
\begin{enumerate}
\item[(a)] The CRS of $\Gamma_b$ is $b$, and $D_\beta \cap \Gamma$ is trivial for any $\beta \in b$.
\item[(b)] One component of $S_b$ is IN for $\Gamma_b$, and another component of $S_b$ is either T or IN for $\Gamma_b$.
\item[(c)] If there is a T component of $\Gamma_b$, denoted by $Q$, then for any $R\in W(Q)$, the group $\theta_R(\Gamma_R)$ is trivial.
\end{enumerate}
Let $\call$, $\calm_1$, $\calm_2$ and $\caln_1$ be subgroupoids of $\calg|_Y$ such that $\calm_2$ is nowhere amenable, $\caln_1$ is amenable and nowhere finite, and we have $\calg_b |_Y < \call$, $\caln_1\vartriangleleft \calm_1$, $\calm_2\vartriangleleft \calm_1$ and $\calm_2 \vartriangleleft \call$.
Then $\calg_b |_Y =\call$.
\end{lem}

\begin{proof}
By Lemma \ref{lem-nor-red} (i) and (ii), all of $\call$, $\calm_1$, $\calm_2$ and $\caln_1$ are reducible.
We will show that the CRS of $\call$ is constant and its value contains $b$.
This implies $\call <\calg_b|_Y$, and the lemma follows.
We may therefore assume that all of the CRS's and the T, IA and IN systems of $\call$, $\calm_1$, $\calm_2$ and $\caln_1$ are constant.
Let $\sigma, \sigma_1, \sigma_2, \tau_1\in \Sigma(S)$ denote the values of the CRS's of $\call$, $\calm_1$, $\calm_2$ and $\caln_1$, respectively.
We set $\calm =\calg_b |_Y$.
The CRS and the T, IA and IN systems of $\calm$ are constant, and their values are the same as those of $\Gamma_b$ by Lemma \ref{lem-red-group}.

We show that $\sigma$ contains $b$.
Assuming to the contrary that $\sigma$ does not contain $b$, we deduce a contradiction.
Suppose that $\sigma$ consists of a single curve of $b$, denoted by $\alpha$.
By the latter condition in condition (a), we can apply Lemma \ref{lem-in-comp}.
Applying it three times, we see that all of $\sigma$, $\sigma_1$, $\sigma_2$ and $\tau_1$ are equal to $\{ \alpha \}$, and the component $S_\alpha$ is IN for any of $\call$, $\calm_1$, $\calm_2$ and $\caln_1$.
By Lemma \ref{lem-no-in}, this contradicts that $\caln_1$ is amenable.
It follows that $\sigma \setminus b$ is non-empty.

As in the second paragraph of the proof of Lemma \ref{lem-s-chain}, picking $\beta \in \sigma \setminus b$, we can show that $\beta$ does not lie in any IN component of $\Gamma_b$, one component of $S_b$ is IN for $\Gamma_b$, another component of $S_b$, denoted by $Q$, is T for $\Gamma_b$, and we have $\beta \in V(Q)$.
We thus have $\sigma \setminus b\in \Sigma(Q)$.

By condition (c), any component of $S_\sigma$ contained in $Q$ is T for $\call$.
Let $R$ denote the component of $S_\sigma$ containing the IN component of $\Gamma_b$.
This $R$ is the unique IN component of $\call$, and at least one curve of $b$ belongs to $V(R)$.

As in the fourth and fifth paragraphs of the proof of Lemma \ref{lem-s-chain}, we can show that $R$ is IN for $\calm_2$ and $\calm_1$ and that $\sigma_2\subset \sigma_1$ and $\sigma_1\setminus b\in \Sigma(Q)$.
At least one curve of $b$ does not belong to $\sigma_1$.
The inclusion $\caln_1\vartriangleleft \calm_1$ implies $\tau_1\subset \sigma_1$ and $\tau_1\setminus b\in \Sigma(Q)\cup \{ \emptyset \}$.
By Lemma \ref{lem-nor-in}, either there exists a T component of $\caln_1$ containing $R$ or $R$ is an IN component of $\caln_1$.
The latter does not hold because $\caln_1$ is amenable.
The former therefore holds.
By condition (c), any component of $S_{\tau_1}$ is T for $\caln_1$.

We claim that $D_{\tau_1}\cap \Gamma$ is trivial.
Condition (c) implies that $\theta_Q(\Gamma_b)$ is trivial.
We have $D_{\tau_1}\cap \Gamma <\Gamma_b<\ker \theta_Q$.
If $\tau_1$ contains no curve of $b$, then we have $\tau_1\in \Sigma(Q)$, and $D_{\tau_1}\cap \Gamma$ is trivial.
Otherwise, denoting by $\gamma$ the curve of $b\cap \sigma_1$, we have $D_{\tau_1}\cap \Gamma <D_\gamma$.
By the latter condition of condition (a), $D_{\tau_1}\cap \Gamma$ is trivial.
The claim was proved.

By the claim, $\ker \theta_{\tau_1}\cap \Gamma$ is trivial. 
Since any component of $S_{\tau_1}$ is T for $\caln_1$, there exists a non-negligible subset $A$ of $Y$ such that $\rho(\caln_1|_A)$ consists of only the neutral element.
This contradicts that $\caln_1$ is nowhere finite.
\end{proof}

\begin{lem}\label{lem-nor-twist}
Pick $\sigma \in \Sigma(S)$.
Suppose that the CRS of $\Gamma_\sigma$ is $\sigma$ and that any component of $S_\sigma$ is either T or IN for $\Gamma_\sigma$.
We set $T_\sigma =D_\sigma \cap \Gamma$ and $\calt_\sigma =T_\sigma \ltimes X$.
Let $\caln$ be an amenable and nowhere finite subgroupoid of $\calg_\sigma |_Y$ with $\caln \vartriangleleft \calg_\sigma |_Y$.
Then, there exists a partition $Y=\bigsqcup_n Y_n$ into countably many measurable subsets such that $\caln|_{Y_n}<\calt_\sigma |_{Y_n}$ for any $n$.
\end{lem}

\begin{proof}
We may assume that all of the CRS and the T, IA and IN systems of $\caln$ are constant.
Let $\tau \in \Sigma(S)$ denote the value of the CRS of $\caln$.
By Lemma \ref{lem-nor-red} (iii), we have $\tau \subset \sigma$.
If there were an IA component of $\caln$, then it would be also IA for $\calg_\sigma |_Y$ by Theorem \ref{thm-ia-max} (ii).
By Lemma \ref{lem-red-group}, this contradicts that $\Gamma_\sigma$ has no IA component.
It follows that $\caln$ has only T components and that there exists a partition $Y=\bigsqcup_n Y_n$ into countably many measurable subsets such that $\rho(\caln|_{Y_n})\subset \ker \theta_{\tau}$ for any $n$.
The lemma was proved.
\end{proof}

\begin{lem}\label{lem-chain}
Suppose that $\Gamma$ is a subgroup of $\mod(S; 3)$ with $\Gamma \vartriangleleft \pmod(S)$.
Let $\calm$ and $\caln$ be subgroupoids of $\calg|_Y$ such that $\caln$ is amenable and nowhere finite, $\calm$ is nowhere amenable, and $\caln \vartriangleleft \calm$.
Note that $\calm$ and $\caln$ are reducible by Lemma \ref{lem-nor-red} (i).
Suppose that all of the CRS's and the T, IA and IN systems of $\calm$ and $\caln$ are constant and that there exists an IA component of $\caln$.
Let $\sigma \in \Sigma(S)$ denote the value of the CRS of $\calm$.
Then, the following assertions hold:
\begin{enumerate}
\item[(i)] There exist subgroupoids $\calm_1$, $\calm_2$ and $\caln_1$ of $\calg|_Y$ such that $\caln_1$ is amenable and nowhere finite, $\calm_2$ is nowhere amenable, and we have $\caln_1\vartriangleleft \calm_1$, $\calm_2\vartriangleleft \calm_1$ and $\calm_2\vartriangleleft \calg_\sigma$.
\item[(ii)] For any non-negligible subset $A$ of $Y$, we have $\calm |_A\neq \calg_\sigma |_A$.
\end{enumerate}
\end{lem}

\begin{proof}
By assumption, we have an IA component of $\caln$, denoted by $Q$.
By Theorem \ref{thm-ia-max} (ii), $Q$ is also IA for $\calm$ and in particular is a component of $S_\sigma$.
By Lemma \ref{lem-ame}, there exists an IN component $R$ of $\calm$.
By Lemma \ref{lem-nor-in}, $R$ is contained in a T component of $\caln$ and in particular is distinct from $Q$.
The groups $\theta_Q(\Gamma_\sigma)$ and $\theta_R(\Gamma_\sigma)$ are non-trivial because $Q$ and $R$ are not T for $\calm$.
By Lemma \ref{lem-group-chain}, there exist subgroups $M_1$, $M_2$ and $N_1$ of $\Gamma$ such that $N_1$ is infinite and amenable, $M_2$ is non-amenable, and we have $N_1\vartriangleleft M_1$, $M_2\vartriangleleft M_1$ and $M_2\vartriangleleft \Gamma_\sigma$.
Setting $\caln_1=(N_1\ltimes X)|_Y$, $\calm_1=(M_1\ltimes X)|_Y$ and $\calm_2=(M_2\ltimes X)|_Y$, we obtain assertion (i).

We prove assertion (ii).
The group $\theta_Q(\Gamma_\sigma)$ is a normal subgroup of some finite index subgroup of $\pmod(Q)$ because $\Gamma \vartriangleleft \mod(S; 3)$.
As is already shown, the group $\theta_Q(\Gamma_\sigma)$ is non-trivial and therefore a non-elementary subgroup of $\pmod(Q)$ by Theorems \ref{thm-pure-t} and \ref{thm-subgr}.
The component $Q$ is IN for $\Gamma_\sigma$ and also IN for $\calg_\sigma$ by Lemma \ref{lem-red-group}.
Assertion (ii) follows because $Q$ is IA for $\calm$.
\end{proof}


\section{Tautness of surface braid groups}\label{sec-sbg}

Two important ingredients in establishing $OE$ rigidity results of surface braid groups are the measure equivalence coupling and the tautness of a group.

\begin{definition}\cite[0.5.E]{Gr93}
Let $\Gamma$ and $\Lambda$ be countable groups.
A \textit{$(\Gamma, \Lambda)$-coupling} is a standard measure space endowed with a $\sigma$-finite measure, $(\Sigma, m)$, on which $\Gamma \times \Lambda$ acts, preserving $m$, so that there exist measurable subsets $X, Y\subset\Sigma$ with $m(X)<\infty$, $m(Y)<\infty$, and
\[\Sigma =\bigsqcup_{\gamma \in \Gamma}(\gamma, e)Y=\bigsqcup_{\lambda \in \Lambda}(e, \lambda)X.\]
The subsets $X$, $Y$ are called \emph{fundamental domains} of the actions $\{ e\} \times \Lambda \ca \Sigma$ and $\Gamma \times \{ e\}\ca \Sigma$, respectively.
 A $(\Gamma, \Gamma)$-coupling is called a \textit{self-coupling} of $\Gamma$.

Two countable groups $\Gamma$, $\Lambda$ are called \textit{measure equivalent ($ME$)} if there exists a $(\Gamma, \Lambda)$-coupling.
\end{definition}

$ME$ defines an equivalence relation between countable groups (\cite[Section 2]{Fu99a}). We refer the reader to \cite{Fu99b} for the relationships between $ME$ and $OE$.

Tautness was introduced in \cite{BFS10, Ki09b} to study $ME$ and $OE$ rigidity aspects of a group from a general standpoint.
While in  \cite{Ki09b}, this property was named differently, the term ``taut" was coined in \cite{BFS10}.

\begin{definition}\label{defn-taut}
Let $G$ be a countable group together with $\Gamma$ a subgroup.
We say that $\Gamma$ is \textit{taut} relative to $G$ if the following conditions hold simultaneously:
\begin{enumerate}
\item[(1)] For any self-coupling $(\Sigma, m)$ of $\Gamma$, there exists a $(\Gamma \times \Gamma)$-equivariant measurable map from $\Sigma$ into $G$, where $\Gamma \times \Gamma$ acts on $G$ by the formula
\[(\gamma_1, \gamma_2)g =\gamma_1g\gamma_2^{-1}\quad \textrm{for}\ \gamma_1, \gamma_2\in \Gamma \ \textrm{and}\ g \in G;\]
\item[(2)] The Dirac measure on the neutral element is the only probability measure on $G$ invariant under conjugation by $\Gamma$.
\end{enumerate}
\end{definition}

Notice that the condition (2) implies the uniqueness of the equivariant map from the condition (1) (\cite[Proposition 4.4]{Fu99a}, \cite[Lemma 5.7]{Ki06}, \cite[Lemma 3.4]{Ki09b} and \cite[Section A.4]{BFS10}).

Let $S=S_{g, k}$ be a surface and denote by $\bar{S}$ the closed surface obtained by attaching a disk to each component of $\partial S$.
We have the homomorphism $\iota \colon \pmod(S)\to \mod(\bar{S})$ induced by the inclusion of $S$ into $\bar{S}$ and denote by $P(S)=\ker \iota$.
If $g\geq 2$ and $k\geq 2$, then $P(S)$ is naturally isomorphic to the pure braid group of $k$ strands on $\bar{S}$ (see \cite[Subsection 3.3]{CKP14} and references therein).
The aim of this section is to prove the following:

\begin{theorem}\label{thm-self}
Let $S=S_{g, k}$ be a surface with $g\geq 2$ and $k\geq 2$.
Then $P(S)$ is taut relative to $\mod^*(S)$.
\end{theorem}

The proof of this theorem will be postponed until Subsection \ref{subsec-taut-sbg} below and will follow in spirit the methods developed in \cite{Ki06}, where it is shown that $\mod^*(S)$ is taut relative to $\mod^*(S)$ for a non-exceptional surface $S$.


\subsection{Complexes for surface braid groups}\label{subsec-pre-sbg}

A {\it holed sphere} is a surface of genus $0$ with non-empty boundary.
Let $S=S_{g, k}$ be a surface.
A curve $\alpha$ in $S$ is called a \textit{hole-bounding curve (HBC)} in $S$ if $\alpha$ is separating in $S$ and cuts off a holed sphere from $S$. 
A pair $\{ \beta, \gamma \}$ of curves in $S$ is called a \textit{hole-bounding pair (HBP)} in $S$ if the following are satisfied:
\begin{itemize}
\item $\beta$ and $\gamma$ are disjoint and non-isotopic;
\item either $\beta$ and $\gamma$ are non-separating in $S$ or $\beta$ and $\gamma$ are separating in $S$ and are not an HBC in $S$; and
\item $S\setminus (\beta \cup \gamma)$ is not connected and has a component of genus zero.
\end{itemize}
An HBP in $S$ is called {\it non-separating} in $S$ if both its curves are non-separating in $S$.
Otherwise it is called {\it separating} in $S$ (see Figure \ref{fig-hbchbp}).
\begin{figure}
\begin{center}
\includegraphics[width=8cm]{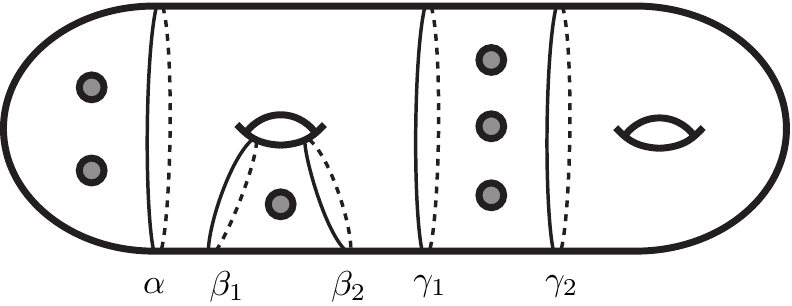}
\caption{The curve $\alpha$ is an HBC, $\{ \beta_1, \beta_2\}$ is a non-separating HBP, and $\{ \gamma_1, \gamma_2\}$ is a separating HBP.}\label{fig-hbchbp}
\end{center}
\end{figure}
If $g=0$, then any curve in $S$ is an HBC.
If $g=1$, then there is no separating HBP in $S$.

The following simplicial complex $\calcp(S)$ was introduced in \cite{KY10a}, inspired by the work of Irmak-Ivanov-McCarthy \cite{IIM03}, to compute virtual automorphisms of $P(S)$.

\medskip

\noindent \textbf{Complexes $\calcp(S)$ and $\calcp_n(S)$.} Let $V_c(S)$ denote the subset of $V(S)$ consisting of isotopy classes of HBCs in $S$.
Let $V_p(S)$ denote the subset of $\Sigma(S)$ consisting of isotopy classes of HBPs in $S$.
We define $\calcp(S)$ as the abstract simplicial complex such that the set of vertices is the disjoint union $V_c(S)\cup V_p(S)$, and a non-empty finite subset $\sigma$ of $V_c(S)\cup V_p(S)$ is a simplex of $\calcp(S)$ if and only if $I(u, v)=0$ for any $u, v\in \sigma$.

Let $V_{np}(S)$ denote the subset of $V_p(S)$ consisting of isotopy classes of non-separating HBPs in $S$.
We define $\calcp_n(S)$ as the full subcomplex spanned by $V_c(S)\cup V_{np}(S)$.

\medskip

Let $X$ and $Y$ be any of $\calcp(S)$ and $\calcp_n(S)$. 
Let $V(X)$ and $V(Y)$ denote the sets of vertices of $X$ and $Y$, respectively. 
Note that a map $\phi \colon V(X)\rightarrow V(Y)$ defines a simplicial map from $X$ into $Y$ if and only if $I(\phi(a), \phi(b))=0$ for any $a, b\in V(X)$ with $I(a, b)=0$. 
A {\it superinjective map} $\phi \colon X\rightarrow Y$ is a simplicial map $\phi \colon X\rightarrow Y$ satisfying $I(\phi(a), \phi(b))\neq 0$, for any $a, b\in V(X)$ with $I(a, b)\neq 0$.  One can show that any superinjective map from $X$ into $Y$ is injective, proceeding as in the proof of \cite[Lemma 3.1]{Ir04}, where any superinjective map from $\calc(S)$ into itself is shown to be injective.

The following theorem will be essential to show tautness of $P(S)$:

\begin{theorem}\label{thm-si}\cite[Corollary 8.15]{KY10b}
Let $S=S_{g, k}$ be a surface with $g\geq 2$ and $k\geq 2$.
Then any superinjective map $\phi \colon \calcp_n(S)\to \calcp(S)$ is induced by an element of $\mod^*(S)$; that is, there exists an $h \in \mod^*(S)$ such that $\phi(v)=h v$ for any vertex $v$ of $\calcp_n(S)$. Moreover, such an $h$ is unique.
\end{theorem}

The uniqueness of $h$ in the theorem follows by a similar argument as in \cite[Lemma 2.2 (i)]{KY10a} and is based on the fact that $\mod^*(S)$ acts faithfully on $\calcp_n(S)$.

For $\alpha \in V(S)$, let $T_\alpha$ denote the subgroup of $\pmod(S)$ generated by the Dehn twist $t_\alpha$ about $\alpha$.
For any BP $b=\{ \beta, \gamma \} \in V_{bp}(S)$, we denote by $T_b$ the subgroup of $\pmod(S)$ generated by $t_\beta t_\gamma^{-1}$.
If $\alpha \in V_c(S)$ and $b\in V_p(S)$, then $T_\alpha$ and $T_b$ are subgroups of $P(S)$.
The family of all groups $T_\alpha$ and $T_b$ with $\alpha \in V_c(S)$ and $b\in V_p(S)$ generates $P(S)$ (\cite[Section 4.1]{Bi74}).
The following result will also be used to show tautness of $P(S)$:

\begin{lem}\label{lem-twist}\cite[Lemma 2.3]{KY10a}
Let $S=S_{g, k}$ be a surface with $g\geq 1$ and $k\geq 0$.
Pick $\sigma \in \Sigma(S)$.
Then, the group $D_\sigma \cap P(S)$ is generated by all $T_\alpha$ and $T_b$ with $\alpha$ an HBC in $\sigma$ and $b$ an HBP of two curves in $\sigma$.
\end{lem}


\subsection{Geometric subgroupoids}

Throughout this subsection, we fix the following notation:
Let $S=S_{g, k}$ be a surface with $g\geq 2$ and $k\geq 2$ and denote by $\Gamma =P(S)\cap \mod(S; 3)$.
Let $\Gamma \ca (X, \mu)$ be a p.m.p.\ action.
Consider $\calg =\Gamma \ltimes X$ and denote by $\rho \colon \calg \to \Gamma$ the canonical projection.
For a set $V$ on which $\Gamma$ acts (e.g., $V(S)$ and $\Sigma(S)$) and for $v \in V$, we denote by $\Gamma_v$ the stabilizer of $v$ in $\Gamma$ and denote by $\calg_v =\Gamma_v \ltimes X$.
For $v \in V_c(S)\cup V_p(S)$, we denote by $\calt_v =(T_v\cap \Gamma)\ltimes X$.

The aim of the following sequence of Lemmas \ref{lem-2-cases}--\ref{lem-p-stab} is to provide a more algebraic description of the subgroupoid $\calg_v$, where $v\in V_c(S)\cup V_p(S)$. Moreover, these lemmas along with Lemma \ref{lem-chain} will be essential for the next subsection.

\begin{lem}\label{lem-2-cases}
Let $Y$ be a non-negligible subset of $X$, $\calm$ a nowhere amenable subgroupoid of $\calg|_Y$, and $\caln$ an amenable and nowhere finite subgroupoid of $\calg|_Y$ with $\caln \vartriangleleft \calm$.
We suppose that the following condition ($\ast$) holds:
\begin{enumerate}
\item[($\ast$)] If $\calm_1$ and $\caln_1$ are subgroupoids of $\calg|_Y$ such that $\caln_1$ is amenable and nowhere finite and we have $\calm <\calm_1$ and $\caln_1\vartriangleleft \calm_1$, then $\calm =\calm_1$.
\end{enumerate} 
Then, there exists a partition $Y=\bigsqcup_nY_n$ into countably many measurable subsets such that for any $n$, one of the following cases (1) and (2) occurs:
\begin{enumerate}
\item[(1)] There exists an element  $v \in V_c(S)\cup V_p(S)$ with $\calm|_{Y_n}=\calg_v|_{Y_n}$.
\item[(2)] All of the CRS's and the T, IA and IN systems of $\calm|_{Y_n}$ and $\caln|_{Y_n}$ are constant, and letting $\sigma \in \Sigma(S)$ denote the value of the CRS of $\calm|_{Y_n}$, we have the following:
No curve in $\sigma$ is an HBC in $S$, no pair of two curves in $\sigma$ is an HBP in $S$, and there exists an IA component of $\caln|_{Y_n}$.
\end{enumerate}
\end{lem}

\begin{proof}
By Lemma \ref{lem-nor-red} (i), $\calm$ and $\caln$ are reducible.
Restricting $\calm$ to a non-negligible subset of $Y$, we may assume that all of the CRS's and the T, IA and IN systems of $\calm$ and $\caln$ are constant.
Let $\sigma \in \Sigma(S)$ be the value of the CRS of $\calm$.

If $\sigma$ contains an HBC $\alpha$ in $S$, then $\calt_\alpha \vartriangleleft \calg_\alpha$ and $\calm <\calg_\alpha |_Y$.
By condition ($\ast$), we have $\calm =\calg_\alpha |_Y$ which gives case (1).
If $\sigma$ contains two curves forming an HBP $b$ in $S$, then similarly we have $\calm =\calg_b|_Y$ which again gives case (1).

Suppose that no curve of $\sigma$ is an HBC in $S$ and no pair of two curves of $\sigma$ is an HBP in $S$.
Let $\tau \in \Sigma(S)$ be the value of the CRS of $\caln$.
Since $\caln \vartriangleleft \calm$, the inclusion $\tau \subset \sigma$ holds by Lemma \ref{lem-nor-red} (iii).
No curve of $\tau$ is an HBC in $S$, and no pair of two curves of $\tau$ is an HBP in $S$.
By Lemma \ref{lem-twist}, $\ker \theta_\tau \cap \Gamma$ is trivial.
If there were no IA component of $\caln$, then every component of $S_\tau$ would be T for $\caln$ by Lemma \ref{lem-no-in}, and $\caln$ would therefore be finite.
This is a contradiction.
It follows that there exists an IA component of $\caln$ and hence case (2) holds.
\end{proof}

\begin{lem}\label{lem-c-stab}
Pick $\alpha \in V_c(S)$.
Let $Y$ be a non-negligible subset of $X$ and set $\calm =\calg_\alpha |_Y$.
Then, the following assertions hold:
\begin{enumerate}
\item[(i)] The group $\Gamma_\alpha$ is non-amenable, and $\calm$ is therefore nowhere amenable.
\item[(ii)] The groupoid $\calm$ fulfills condition ($\ast$) in Lemma \ref{lem-2-cases}.
\item[(iii)] Let $\call$, $\calm_1$, $\calm_2$ and $\caln_1$ be subgroupoids of $\calg|_Y$ such that $\calm_2$ is nowhere amenable, $\caln_1$ is amenable and nowhere finite, and we have $\calm < \call$, $\caln_1\vartriangleleft \calm_1$, $\calm_2\vartriangleleft \calm_1$ and $\calm_2 \vartriangleleft \call$.
Then $\calm =\call$.
\item[(iv)] If $\caln$ is an amenable and nowhere finite subgroupoid of $\calm$ with $\caln \vartriangleleft \calm$, then there exists a partition $Y=\bigsqcup_nY_n$ into countably many measurable subsets such that $\caln|_{Y_n}<\calt_\alpha|_{Y_n}$ for any $n$.
\end{enumerate}
\end{lem}

\begin{proof}
First we check that conditions (a)--(c) in Lemma \ref{lem-s-chain} hold.
The HBC $\alpha$ divides $S$ into a component of genus 0 and a component of genus $g$, denoted by $Q$ and $R$, respectively.
If $Q$ is a pair of pants, then $R$ contains a non-separating HBP in $S$, and $\theta_R(\Gamma_\alpha)$ is therefore non-elementary.
If $Q$ is not a pair of pants, then $Q$ contains an HBC in $S$, and $R$ contains a non-separating HBP in $S$.
Since both  $\theta_Q(\Gamma_\alpha)$ and $\theta_R(\Gamma_\alpha)$ are non-elementary then the conditions (a)--(c) in Lemma \ref{lem-s-chain} follow.

Assertion (i) holds because there exists an IN component of $\Gamma_\alpha$.
We prove assertion (ii).
To check condition ($\ast$) in Lemma \ref{lem-2-cases}, let $\calm_1$ and $\caln_1$ be subgroupoids of $\calg|_Y$ such that $\caln_1$ is an amenable and nowhere finite and we have $\calm <\calm_1$ and $\caln_1\vartriangleleft \calm_1$.
Putting $\call =\calm_1=\calm_2$ and applying Lemma \ref{lem-s-chain}, we obtain $\calm =\calm_1$ and hence assertion (ii) follows.
Finally, assertion (iii) follows directly from Lemma \ref{lem-s-chain}.
and assertion (iv) follows from Lemma \ref{lem-nor-twist}.
\end{proof}

\begin{lem}\label{lem-p-stab}
Pick $b\in V_{np}(S)$.
Let $Y$ be a non-negligible subset of $X$ and set $\calm =\calg_b |_Y$.
Then, the following assertions hold:
\begin{enumerate}
\item[(i)] The group $\Gamma_b$ is non-amenable, and hence $\calm$ is nowhere amenable.
\item[(ii)] The groupoid $\calm$ fulfills condition ($\ast$) in Lemma \ref{lem-2-cases}.
\item[(iii)] Let $\call$, $\calm_1$, $\calm_2$ and $\caln_1$ be subgroupoids of $\calg|_Y$ such that $\calm_2$ is nowhere amenable, $\caln_1$ is amenable and nowhere finite, and we have $\calm < \call$, $\caln_1\vartriangleleft \calm_1$, $\calm_2\vartriangleleft \calm_1$ and $\calm_2 \vartriangleleft \call$.
Then $\calm =\call$.
\end{enumerate}
\end{lem}

\begin{proof}
First we check that conditions (a)--(c) in Lemma \ref{lem-bp-chain} hold.
Let $\bar{S}$ denote the closed surface obtained by filling a disk to each component of $\partial S$.
For any $\beta \in b$, the group $D_\beta \cap \Gamma$ is trivial because $\beta$ is also essential in $\bar{S}$ through the inclusion of $S$ into $\bar{S}$.

The HBP $b$ divides $S$ into a component of genus 0 and a component of genus $g-1$, denoted by $Q$ and $R$, respectively.
If $Q$ is a pair of pants, then $R$ contains a component of $\partial S$ and contains a curve in $S$ forming an HBP in $S$ together with a curve of $b$.
Thus we have that $\theta_R(\Gamma_b)$ is non-elementary and hence the conditions (a)--(c) in Lemma \ref{lem-bp-chain} follow.

Suppose that $Q$ is not a pair of pants.
There is an HBC in $S$ which belongs to $V(Q)$, and $\theta_Q(\Gamma_b)$ is thus non-elementary.
If $R$ contains a component of $\partial S$, then $R$ contains a curve in $S$ forming an HBP in $S$ together with a curve of $b$, and $\theta_R(\Gamma_b)$ is non-elementary.
Conditions (a)--(c) in Lemma \ref{lem-bp-chain} follow.

Suppose that $R$ contains no component of $\partial S$.
The set $V(R)$ is injectively embedded into $V(\bar{S})$ through the inclusion of $S$ into $\bar{S}$.
For any $R_1\in W(R)$, the set $V(R_1)$ is also injectively embedded into $V(\bar{S})$, and the group $\theta_{R_1}(\Gamma_{R_1})$ acts on $V(R_1)$ trivially because any element of $\Gamma$ acts on $V(\bar{S})$ trivially.
By Theorem \ref{thm-pure-t} and Lemma \ref{lem-curve-stab}, the group $\theta_{R_1}(\Gamma_{R_1})$ is trivial.
Conditions (a)--(c) in Lemma \ref{lem-bp-chain} follow.

Assertion (i) of the lemma holds because there exists an IN component of $\Gamma_b$.
Assertions (ii) and (iii) follow from Lemma \ref{lem-bp-chain}, proceeding in a similar manner as in the proofs of parts (ii) and (iii) of Lemma \ref{lem-c-stab} above.
\end{proof}

The non-separability assumption on the HBP $b$ in Lemma \ref{lem-p-stab} was necessary because of the use of Lemma \ref{lem-bp-chain} in its proof.
However, the following lemma holds for any HBP in $S$ and it follows from Lemma \ref{lem-nor-twist}.

\begin{lem}\label{lem-bp-nor}
Pick $b\in V_p(S)$.
Let $Y$ be a non-negligible subset of $X$ and set $\calm =\calg_b |_Y$.
If $\caln$ is an amenable and nowhere finite subgroupoid of $\calm$ with $\caln \vartriangleleft \calm$, then there exists a partition $Y=\bigsqcup_nY_n$ into countably many measurable subsets such that $\caln|_{Y_n}<\calt_b|_{Y_n}$ for any $n$.
\end{lem}

\begin{Remark}
For some separating HBP in $S$, Lemma \ref{lem-p-stab} (iii) does not hold.
To see this, let $b=\{ \beta, \gamma \}$ be a separating HBP in $S$ such that at least one of the two components of $S_b$ of positive genus contains a component of $\partial S$.
Let $Q_1$, $Q_2$ and $Q_3$ denote the components of $S_b$ so that $Q_1$ and $Q_3$ are of positive genus, and $Q_1$ contains a component of $\partial S$.
We may assume that $\beta$ is a component of $\partial Q_1$.
Then, there exist subgroups $L$, $M_1$, $M_2$ and $N_1$ of $\Gamma$ such that $M_2$ is non-amenable, $N_1$ is infinite and amenable, we have $\Gamma_b<L$, $N_1\vartriangleleft M_1$, $M_2\vartriangleleft M_1$ and $M_2\vartriangleleft L$, and we have $[L: \Gamma_b]=\infty$.
In fact, letting $R$ be the component of $S_\beta$ containing $Q_2$ and $Q_3$ and setting
\[L=\Gamma_\beta,\quad M_2=\ker \theta_R\cap \Gamma, \quad N_1=T_b\cap \Gamma \quad \textrm{and}\quad M_1=M_2\vee N_1,\]
we can check that these groups satisfy the desired properties.
\end{Remark}


\subsection{Proof of tautness}\label{subsec-taut-sbg}

Throughout this subsection, except for the proof of Theorem \ref{thm-self}, we fix the following notation:
Let $S=S_{g, k}$ be a surface with $g\geq 2$ and $k\geq 2$.
We set $\Gamma =P(S)\cap \mod(S; 3)$.
We also set $V=V_c(S)\cup V_p(S)$ and $V_0=V_c(S)\cup V_{np}(S)$.

Let $\Gamma \ca (X, \mu)$ and $\Gamma \ca (Y, \nu)$ be p.m.p.\ actions, and denote by $\calg =\Gamma \ltimes X$ and $\calh =\Gamma \ltimes Y$.
Suppose that we have non-negligible subsets $A\subset X$ and $B\subset Y$ with $\Gamma A=X$ and $\Gamma B=Y$, and have an isomorphism $f\colon \calg|_A\to \calh|_B$.

As before, for any set $L$ on which $\Gamma$ acts and for any element $l\in L$, we denote by $\Gamma_l$ the stabilizer of $l$ in $\Gamma$ and we set $\calg_l =\Gamma_l\ltimes X$ and $\calh_l=\Gamma_l\ltimes Y$. For $v\in V$, we denote by $\calt_v=(T_v\cap \Gamma)\ltimes X$ and $\calu_v=(T_v\cap \Gamma)\ltimes Y$.

\begin{lem}\label{lem-stab-pre}
For any $v\in V_0$, there exist a countable set $N$, a partition $A=\bigsqcup_{n\in N}A_n$ into measurable subsets and $w_n\in V$ indexed by $n\in N$ with $f(\calg_v|_{A_n})=\calh_{w_n}|_{f(A_n)}$ for any $n\in N$. 
\end{lem}

\begin{proof}
Fix $v\in V_0$.
By Lemma \ref{lem-c-stab} (i) and Lemma \ref{lem-p-stab} (i), $\calg_v$ is nowhere amenable.
By Lemma \ref{lem-c-stab} (ii) and Lemma \ref{lem-p-stab} (ii), $\calg_v|_A$ satisfies condition ($\ast$) in Lemma \ref{lem-2-cases}, and so does the image $f(\calg_v)|_B$.
We set $\calm =f(\calg_v)|_B$ and $\caln =f(\calt_v)|_B$.
Applying Lemma \ref{lem-2-cases} to $\caln \vartriangleleft \calm$, we have a partition $B=\bigsqcup_n B_n$ into countably many measurable subsets such that for any $n$, one of the following cases occurs:
\begin{enumerate}
\item[(1)] There exists a $w \in V$ with $\calm|_{B_n}=\calh_w|_{B_n}$.
\item[(2)] All of the CRS's and the T, IA and IN systems of $\calm|_{B_n}$ and $\caln|_{B_n}$ are constant, and letting $\sigma \in \Sigma(S)$ denote the value of the CRS of $\calm|_{B_n}$, we have the following:
No curve in $\sigma$ is an HBC in $S$, no pair of two curves in $\sigma$ is an HBP in $S$, and there exists an IA component of $\caln|_{B_n}$.
\end{enumerate}
The lemma follows if for any $n$, case (2) never occurs. 
Assume that there is an $n$ for which case (2) occurs.
By Lemma \ref{lem-chain}, we have subgroupoids $\call$, $\calm_1$, $\calm_2$ and $\caln_1$ of $\calh|_{B_n}$ such that $\calm_2$ is nowhere amenable, $\caln_1$ is amenable and nowhere finite, and we have $\calm < \call$, $\caln_1\vartriangleleft \calm_1$, $\calm_2\vartriangleleft \calm_1$ and $\calm_2 \vartriangleleft \call$, and for any non-negligible subset $C$ of $B_n$, we have $\calm|_C\neq \call|_C$.
The image of $\call$, $\calm_1$, $\calm_2$ and $\caln_1$ under $f^{-1}$ also has the same property.
Applying Lemma \ref{lem-c-stab} (iii) and Lemma \ref{lem-p-stab} (iii) to the groupoid $\calg_v|_{f^{-1}(B_n)}=f^{-1}(\calm |_{B_n})$ and the images of $\call$, $\calm_1$, $\calm_2$ and $\caln_1$ under $f^{-1}$, we obtain the equation $\calg_v|_{f^{-1}(B_n)}=f^{-1}(\call)|_{f^{-1}(B_n)}$.
This is a contradiction.
\end{proof}

\begin{lem}\label{lem-twist-pre}
For any $v\in V_0$, there exist a countable set $N$, a partition $A=\bigsqcup_{n\in N}A_n$ into measurable subsets and $w_n\in V$ indexed by $n\in N$ with $f(\calt_v|_{A_n})=\calu_{w_n}|_{f(A_n)}$ for any $n\in N$. 
\end{lem}

\begin{proof}
Fix $v\in V_0$. By Lemma \ref{lem-stab-pre}, we have a countable set $N$, a partition $A=\bigsqcup_{n\in N}A_n$ into measurable subsets and $w_n\in V$ indexed by $n\in N$ with $f(\calg_v|_{A_n})=\calh_{w_n}|_{f(A_n)}$ for any $n\in N$.
We have $\calt_v\vartriangleleft \calg_v$ and thus $f(\calt_v|_{A_n})\vartriangleleft \calh_{w_n}|_{f(A_n)}$ for any $n$.
By Lemma \ref{lem-c-stab} (iv) and Lemma \ref{lem-bp-nor}, taking a finer partition of $A$, we may suppose that for any $n$, the inclusion $f(\calt_v|_{A_n})<\calu_{w_n}|_{f(A_n)}$ holds.
Applying Lemma \ref{lem-c-stab} (iv) and Lemma \ref{lem-bp-nor} to $f^{-1}(\calu_{w_n})|_{A_n}$, we can find a finer partition of $A$ so that for any $n$, the equation $f(\calt_v|_{A_n})=\calu_{w_n}|_{f(A_n)}$ holds.
\end{proof}

Note that for any non-negligible subset $B_1\subset B$ and for any $w_1, w_2\in V$, the equation $\calu_{w_1}|_B=\calu_{w_2}|_B$ implies $w_1=w_2$.
This follows from that for any distinct $v_1, v_2\in V$, the intersection $T_{v_1}\cap T_{v_2}$ is trivial.

We define a map $\varphi_0\colon A\times V_0\to V$ as follows:
Pick $v\in V_0$.
By Lemma \ref{lem-twist-pre}, we have a partition $A=\bigsqcup_n A_n$ into measurable subsets and $w_n\in V$ with $f(\calt_v|_{A_n})=\calu_{w_n}|_{f(A_n)}$ for any $n$.
We set $\varphi_0(x, v)=w_n$ for $x\in A_n$.
This definition is independent of the choice of the partition of $A$ because of the assertion in the previous paragraph.

\begin{lem}
For almost every $x\in A$, the map $\varphi_0(x, \cdot)\colon V_0\to V$ defines a superinjective map from $\calcp_n(S)$ into $\calcp(S)$.
\end{lem}

\begin{proof}
Pick $v_1, v_2\in V_0$ and a non-negligible subset $Z$ of $A$ such that for any $i=1, 2$, we have $f(\calt_{v_i}|_Z)=\calu_{w_i}|_{f(Z)}$ with some $w_i\in V$.
Suppose first that $I(v_1, v_2)=0$.
The groupoid $(\calt_{v_1}|_Z)\vee (\calt_{v_2}|_Z)$ is amenable, and so is its image under $f$, $(\calu_{w_1}|_{f(Z)})\vee (\calu_{w_2}|_{f(Z)})$.
If we had $I(w_1, w_2)\neq 0$, then for any sufficiently large $n, m\in \mathbb{N}$, $t_{w_1}^n$ and $t_{w_2}^m$ would generate a free group of rank $2$.
This is proved similarly to \cite[Lemma 5.3]{Ki06} through \cite[Theorem 4.3]{Iv92}.
By \cite[Lemma 3.20]{Ki06}, $(\calu_{w_1}|_{f(Z)})\vee (\calu_{w_2}|_{f(Z)})$ is non-amenable.
This is a contradiction.
We therefore have $I(w_1, w_2)=0$.
A verbatim argument shows that if $I(w_1, w_2)=0$, then $I(v_1, v_2)=0$.
Thus the statement follows.
\end{proof}

By Theorem \ref{thm-si}, for almost every $x\in A$, there exists a unique $h \in \mod^*(S)$ such that for any $v\in V_0$, the equation $\varphi_0(x, v)=h v$ holds.
We define $\varphi(x)\in \mod^*(S)$ to be this $h$ and obtain the map $\varphi \colon A\to \mod^*(S)$, which is measurable.

\begin{lem}\label{lem-cocycle}
Let $\eta \colon \calg|_A\to \Gamma$ denote the composition of the isomorphism $f\colon \calg|_A\to \calh|_B$ with the projection from $\calh$ onto $\Gamma$.
Then, for any $\gamma \in \Gamma$ and almost every $x\in A\cap \gamma^{-1}A$, we have
\[\eta(\gamma, x)=\varphi(\gamma x)\gamma \varphi(x)^{-1}.\]
\end{lem}

\begin{proof}
We proceed in a similar manner as in the proof of \cite[Lemma 5.5]{Ki06}. Pick $\gamma \in \Gamma$ and $v\in V_0$. It suffices to show that the equation
\[\eta(\gamma, x)\varphi(x)v=\varphi(\gamma x)\gamma v\]
holds, for almost every $x\in A\cap \gamma^{-1}A$.
Consider a non-negligible subset $Z\subset A$ satisfying the following: $\gamma Z\subset A$; there exists $\delta \in \Gamma$ with $\eta(\gamma, x)=\delta$ for any $x\in Z$; the map $\varphi$ is constant on $Z$ and on $\gamma Z$ and denote by $s, t\in \mod^*(S)$ its constant values, respectively; and we have $f(\calt_v|_Z)=\calu_{sv}|_{f(Z)}$ and $f(\calt_{\gamma v}|_{\gamma Z})=\calu_{t\gamma v}|_{f(\gamma Z)}$.
The set $A\cap \gamma^{-1}A$ is covered by countably many such subsets as $Z$.

For $h\in \Gamma$, let $\textrm{Ad}h$ be the inner automorphism of $\calg$ defined by
\[\textrm{Ad}h(l, x)=(h, lx)(l, x)(h, x)^{-1}=(hlh^{-1}, hx)\quad \textrm{for}\ (l, x)\in \calg.\]
For any $\alpha \in V(S)$ and $h\in \mod^*(S)$, the equation $ht_\alpha h^{-1}=t_{h\alpha}^\varepsilon$ holds, where $\varepsilon =1$ if $h\in \mod(S)$, and $\varepsilon =-1$ otherwise (\cite[Lemma 4.1.C]{Iv02}).
It follows that $\textrm{Ad}\gamma (\calt_v|_Z)=\calt_{\gamma v}|_{\gamma Z}$.
Applying $f$ to this equation, we get that $\textrm{Ad}\delta (\calu_{sv}|_{f(Z)})=\calu_{t\gamma v}|_{f(\gamma Z)}$.
The left hand side equals $\calu_{\delta sv}|_{\delta f(Z)}$. Hence $\delta sv=t\gamma v$ holds which in turn gives the desired equation.
\end{proof}

\begin{proof}[Proof of Theorem \ref{thm-self}]
The proof uses the connection between ME couplings and stable orbit equivalence, revealed in \cite{Fu99b}.
Our aim is to show that $P(S)$ is taut relative to $\mod^*(S)$.
We check conditions (1) and (2) in Definition \ref{defn-taut}.
Condition (2) follows from Lemma \ref{lem-icc}.
To check condition (1), let $(\Sigma, m)$ be a self-coupling of $P(S)$.
We set $\Gamma =P(S)\cap \mod(S; 3)$.
The space $(\Sigma, m)$ is also a self-coupling of $\Gamma$.
To distinguish two $\Gamma$'s, we put $\Lambda =\Gamma$ and regard $(\Sigma, m)$ as a $(\Gamma, \Lambda)$-coupling. 
Pick measurable fundamental domains $X, Y\subset \Sigma$ for the actions $\{ e\} \times \Lambda \ca \Sigma$ and $\Gamma \times \{ e\} \ca \Sigma$, respectively.
Let $\eta \colon \Gamma \times X\to \Lambda$ be the cocycle defined so that $(\gamma, \eta(\gamma, x))x\in X$ for $\gamma \in \Gamma$ and $x\in X$.
The map $(\gamma, x)\mapsto (\gamma, \eta(\gamma, x))x$ defines the natural action of $\Gamma$ on $X$, which is p.m.p.\ with respect to the restriction of $m$ to $X$.
To distinguish this action from the original action of $\Gamma \times \{ e\}$ on $\Sigma$, we denote $(\gamma, \eta(\gamma, x))x$ by $\gamma \cdot x$, using a dot.
We also have the natural p.m.p.\ action of $\Lambda$ on $Y$.
We set $\calg =\Gamma \ltimes X$ and $\calh =\Lambda \ltimes Y$.

By \cite[Lemma 2.27]{Ki09a}, we can choose $X$ and $Y$ so that, putting $Z=X\cap Y$, we have $\Gamma \cdot Z=X$ and $\Lambda \cdot Z=Y$.
By \cite[Proposition 2.29]{Ki09a}, we have the isomorphism $f\colon \calg|_Z\to \calh|_Z$ defined by $f(\gamma, x)=(\eta(\gamma, x), x)$ for $(\gamma, x)\in \calg|_Z$.
By Lemma \ref{lem-cocycle}, there exists a measurable map $\varphi \colon Z\to \mod^*(S)$ satisfying the equation \begin{equation}\label{eqeta}\eta(\gamma, x)=\varphi(\gamma x)\gamma \varphi(x)^{-1}\quad \textrm{ for almost every } (\gamma, x)\in \calg|_Z.\end{equation}

We define a map $\Phi \colon \Sigma \to \mod^*(S)$ by
\[\Phi((\gamma, \lambda)x)=\gamma \varphi(x)^{-1}\lambda^{-1}\quad \textrm{for}\ \gamma \in \Gamma,\ \lambda \in \Lambda \ \textrm{and}\ x\in Z.\]
Using the same arguments as in the proofs of \cite[Theorem 5.6]{Ki06} and \cite[Theorem 4.4]{Ki09b}, the formula (\ref{eqeta}) above shows that $\Phi$ is well-defined.
By definition, the map $\Phi$ is $(\Gamma \times \Lambda)$-equivariant.

The group $\Gamma =\Lambda$ is a finite index subgroup of $P(S)$.
As $\Gamma =\Lambda$ is a normal subgroup of the finite index subgroup $\mod(S; 3)$ of $\mod(S)$, by Lemma \ref{lem-icc}, the Dirac measure on the neutral element is the only probability measure on $\mod^*(S)$ invariant under conjugation by $\Gamma =\Lambda$.
By these two facts, \cite[Lemma 5.8]{Ki06} is applicable to $\Phi$, and it implies that $\Phi$ is $(P(S)\times P(S))$-equivariant. Hence, condition (1) in Definition \ref{defn-taut} holds.
\end{proof}


\section{Tautness of the Torelli group and the Johnson kernel}\label{sec-t-j}

Let $S=S_{g, k}$ be a surface. As in Subsection \ref{subsec-pre-sbg} above, for $\alpha \in V(S)$, let $t_\alpha \in \pmod(S)$ denote the Dehn twist about $\alpha$ and $T_\alpha$ denote the group generated by $t_\alpha$.
For a BP $b=\{ \beta, \gamma \}$ in $S$, let $T_b$ denote the group generated by $t_\beta t_\gamma^{-1}$.
We define the \textit{Torelli group} $\cali(S)$ as the group generated by all $T_\alpha$ and $T_b$ with $\alpha \in V_s(S)$ and $b\in V_{bp}(S)$.
We define the \textit{Johnson kernel} $\calk(S)$ as the group generated by all $T_\alpha$ with $\alpha \in V_s(S)$.
We refer to \cite[Chapter 6]{FM11} for background of these groups.
The main result of this section is the following:

\begin{theorem}\label{thm-taut-t-j}
Let $S=S_{g, k}$ be a surface.
Then, the following assertions hold:
\begin{enumerate}
\item[(i)] If either $g=1$ and $k\geq 3$; $g=2$ and $k\geq 1$; or $g\geq 3$ and $k\geq 0$, then $\cali(S)$ is taut relative to $\mod^*(S)$.
\item[(ii)] If either $g=1$ and $k\geq 3$; $g=2$ and $k\geq 2$; or $g\geq 3$ and $k\geq 0$, then $\calk(S)$ is taut relative to $\mod^*(S)$.
\end{enumerate}
\end{theorem}

The proof of this theorem follows the same outline as the proof of Theorem \ref{thm-self} above and therefore many details will be omitted. We will only state several keys lemmas in Subsections \ref{subsec-t} and \ref{subsec-j} which, in the same spirit as in the previous sections, provide an algebraic description of various geometric subgroupoids.

\begin{Remark}\label{identification}
We briefly discuss the cases not covered by Theorem \ref{thm-taut-t-j}.
Suppose $g=0$.
We have $V_c(S)=V_s(S)=V(S)$ and $P(S)=\calk(S)=\cali(S)=\pmod(S)$.
If $k\geq 5$, then  by \cite[Corollary 5.9]{Ki06} $\pmod(S)$ is taut with respect to $\mod^*(S)$.
If $k=4$, then $\pmod(S)$ is commensurable with a free group of rank $2$.
If $k\leq 3$, then $\pmod(S)$ is trivial.

Suppose $g=1$.
We have $V_c(S)=V_s(S)$, $V_p(S)=V_{bp}(S)$ and $P(S)=\cali(S)$.
If $k\leq 1$, then both $\cali(S)$ and $\calk(S)$ are trivial.
If $k=2$, then $\cali(S)$ and $\calk(S)$ are isomorphic to the free groups of rank $2$ and $\infty$, respectively, by Birman's exact sequence.

Suppose $g=2$.
If $k=0$, then $V_{bp}(S)$ is empty, we have $\calk(S)=\cali(S)$, and the group $\cali(S)$ is isomorphic to the free group of rank $\infty$ (\cite{Me92, BBM10}).
If $k=1$, then $\calk(S)$ is not a free group because there exist two separating curves in $S$ which are disjoint and non-isotopic.
In this case we do not know whether $\calk(S)$ is taut relative to $\mod^*(S)$. The main difficulty here stems from the fact that the complex of separating curves, $\calc_s(S)$, defined in Subsection \ref{subsec-cpx-t-j}, has simplicial automorphisms which are not induced by an element of $\mod^*(S)$ (see \cite[Remark 1.3]{Ki09c}).
\end{Remark}


\subsection{Complexes for the Torelli group and the Johnson kernel}\label{subsec-cpx-t-j}

We consider the following versions of curves complexes:

\medskip

\noindent \textbf{Complexes $\calt(S)$ and $\calc_s(S)$.}
We define $\calt(S)$ as the abstract simplicial complex so that the set of vertices is the disjoint union $V_s(S)\cup V_{bp}(S)$, and a non-empty finite subset $\sigma$ of $V_s(S)\cup V_{bp}(S)$ is a simplex of $\calt(S)$ if and only if $I(u, v)=0$ for any $u, v\in \sigma$.
The complex $\calt(S)$ is called the \textit{Torelli complex} of $S$.

We define $\calc_s(S)$ as the full subcomplex of $\calc(S)$ spanned by $V_s(S)$.
This is identified with the full subcomplex of $\calt(S)$ spanned by $V_s(S)$.
The complex $\calc_s(S)$ is called the \textit{complex of separating curves} of $S$.

\medskip

The Torelli complex was examined by Farb-Ivanov \cite{FI05}, McCarthy-Vautaw \cite{MV03} and Brendle-Margalit \cite{BM04, BM08} to compute virtual automorphisms of the Torelli group.
They have dealt mostly with closed surfaces, and subsequently Yamagata and the second author have generalized their results to the following (\cite[Theorem 1.1]{Ki09c}, \cite[Theorem 1.2]{KY10c}):

\begin{theorem}\label{thm-t-aut}
Let $S=S_{g, k}$ be a surface and suppose either $g=1$ and $k\geq 3$; $g=2$ and $k\geq 1$; or $g\geq 3$ and $k\geq 0$.
Then, any simplicial automorphism $\phi$ of $\calt(S)$ is induced by an element of $\mod^*(S)$, that is, there exists an $h\in \mod^*(S)$ such that $\phi(v)=h v$ for any vertex $v$ of $\calt(S)$.
Moreover, such an $h$ is unique.
\end{theorem}

The complex of separating curves was examined by Farb-Ivanov \cite{FI05} and Brendle-Margalit \cite{BM04, BM08} to compute virtual automorphisms of the Johnson kernel.
They have dealt mostly with closed surfaces, and subsequently the second author generalized their results to the following (\cite[Theorem 1.2]{Ki09c}):

\begin{theorem}\label{thm-cs-aut}
Let $S=S_{g, k}$ be a surface and suppose either $g=1$ and $k\geq 3$; $g=2$ and $k\geq 2$; or $g\geq 3$ and $k\geq 0$.
Then, any simplicial automorphism $\phi$ of $\calc_s(S)$ is induced by an element of $\mod^*(S)$, that is, there exists an $h \in \mod^*(S)$ such that $\phi(v)=h v$ for any vertex $v$ of $\calc_s(S)$.
Moreover, such an $h$ is unique.
\end{theorem}

Assertions (i) and (ii) of the next theorem were obtained in \cite{Vau02} and \cite{BBM10}, respectively, in the context of for closed surfaces. Relying on those assertions, the second author obtained the following generalization:

\begin{theorem}\label{thm-t-j-twist}\cite[Theorem 6.1]{Ki09c}
Let $S=S_{g, k}$ be a surface satisfying $3g+k-4>0$ and fix $\sigma \in \Sigma(S)$.
Then, the following assertions hold:
\begin{enumerate}
\item[(i)] The group $D_\sigma \cap \cali(S)$ is generated by all $T_\alpha$ and $T_b$, where $\alpha \in \sigma \cap V_s(S)$, $b\in V_{bp}(S)$ and $b\subset \sigma$.
\item[(ii)] The group $D_\sigma \cap \calk(S)$ is generated by all $T_\alpha$, where $\alpha \in \sigma \cap V_s(S)$.
\end{enumerate}
\end{theorem}

The following lemma will be used in Subsection \ref{subsec-t}:

\begin{lem}\label{lem-handle}
Let $S=S_{g, k}$ be a surface with $3g+k-4>0$.
Let $\alpha \in V(S)$ be a separating curve in $S$ cutting off a handle $Q$ from $S$.
Then $\theta_Q(\cali(S)_\alpha)$ is trivial.
\end{lem}

\begin{proof}
Let $\bar{S}$ be the closed surface obtained by attaching a disk to each component of $\partial S$.
The group $\cali(S)$ acts on $H_1(\bar{S}, \mathbb{Z})$ trivially.
The group $\theta_Q(\cali(S)_\alpha)$ acts on $H_1(Q, \mathbb{Z})$ trivially because $H_1(Q, \mathbb{Z})$ is a subgroup of $H_1(\bar{S}, \mathbb{Z})$ through the inclusion of $Q$ into $\bar{S}$. Then the lemma follows because $\mod(Q)$ acts on $H_1(Q, \mathbb{Z})$ faithfully.
\end{proof}


\subsection{Geometric subgroupoids for the Torelli group}\label{subsec-t}

Throughout this subsection, we fix the following notation:
Let $S=S_{g, k}$ be a surface with either $g=1$ and $k\geq 3$; $g=2$ and $k\geq 1$; or $g\geq 3$ and $k\geq 0$.
Denote by $\Gamma =\cali(S)\cap \mod(S; 3)$ and let $\Gamma \ca (X, \mu)$ be a p.m.p.\ action.
Define $\calg =\Gamma \ltimes X$ and denote by $\rho \colon \calg \to \Gamma$ the canonical projection.
For any set $V$ on which $\Gamma$ acts and for any element $v \in V$, we denote by $\Gamma_v$ the stabilizer of $v$ in $\Gamma$ and we let $\calg_v =\Gamma_v \ltimes X$.
For every $v \in V_s(S)\cup V_{bp}(S)$ we set $\calt_v =(T_v\cap \Gamma)\ltimes X$.

\begin{lem}\label{lem-2-cases-t}
Let $Y$ be a non-negligible subset of $X$, $\calm$ a nowhere amenable subgroupoid of $\calg|_Y$, and $\caln$ an amenable and nowhere finite subgroupoid with $\caln \vartriangleleft \calm$.
Suppose the following condition holds:
\begin{enumerate}
\item[($\ast$)] If $\calm_1$ and $\caln_1$ are subgroupoids of $\calg|_Y$ such that $\caln_1$ is amenable and nowhere finite and we have $\calm <\calm_1$ and $\caln_1\vartriangleleft \calm_1$, then $\calm =\calm_1$.
\end{enumerate} 
Then, there exists a partition $Y=\bigsqcup_nY_n$ into countably many measurable subsets such that for any $n$, one of the following cases occurs:
\begin{enumerate}
\item[(1)] There exists a $v \in V_s(S)\cup V_{bp}(S)$ with $\calm|_{Y_n}=\calg_v|_{Y_n}$.
\item[(2)] All of the CRS's and the T, IA and IN systems of $\calm|_{Y_n}$ and $\caln|_{Y_n}$ are constant, and letting $\sigma \in \Sigma(S)$ to be the value of the CRS of $\calm|_{Y_n}$, we have the following:
No curve in $\sigma$ is separating in $S$, no pair of two curves in $\sigma$ is a BP in $S$, and there exists an IA component of $\caln|_{Y_n}$.
\end{enumerate}
\end{lem}

\begin{proof}
The proof will be omitted as it follows verbatim as the proof of Lemma \ref{lem-2-cases} above. We only remark that in the course of the proof, one has to use Theorem \ref{thm-t-j-twist} (i) in place of Lemma \ref{lem-twist}.
\end{proof}

\begin{lem}\label{lem-i-stab}
Fix $v\in V_s(S)\cup V_{bp}(S)$.
Let $Y$ be a non-negligible subset of $X$ and set $\calm =\calg_v |_Y$.
Then, the following assertions hold:
\begin{enumerate}
\item[(i)] The group $\Gamma_v$ is non-amenable, and $\calm$ is therefore nowhere amenable.
\item[(ii)] The groupoid $\calm$ satisfies condition ($\ast$) in Lemma \ref{lem-2-cases-t}.
\item[(iii)] Let $\call$, $\calm_1$, $\calm_2$ and $\caln_1$ be subgroupoids of $\calg|_Y$ such that $\calm_2$ is nowhere amenable, $\caln_1$ is amenable and nowhere finite, and we have $\calm < \call$, $\caln_1\vartriangleleft \calm_1$, $\calm_2\vartriangleleft \calm_1$ and $\calm_2 \vartriangleleft \call$.
Then $\calm =\call$.
\item[(iv)] If $\caln$ is an amenable and nowhere finite subgroupoid of $\calm$ with $\caln \vartriangleleft \calm$, then there exists a partition $Y=\bigsqcup_nY_n$ into countably many measurable subsets such that $\caln|_{Y_n}<\calt_v|_{Y_n}$ for any $n$.
\end{enumerate}
\end{lem}

\begin{proof}
Let $Q$ be a component of $S_v$.
Either $Q$ is a pair of pants, $Q$ is a handle, or $\theta_Q(\Gamma_v)$ is non-elementary.
If $Q$ is a handle, then $v$ is a separating curve in $S$ cutting off $Q$ from $S$, and for any $R\in W(Q)$, we have $\partial Q\subset \partial R$.
It follows that for any $R\in W(Q)$, we have $\Gamma_R<\Gamma_v$.
By Lemma \ref{lem-handle}, $\theta_Q(\Gamma_v)$ is trivial, and so is $\theta_R(\Gamma_R)$.

It follows that if $v\in V_s(S)$, then conditions (a)--(c) in Lemma \ref{lem-s-chain} hold.
If $v\in V_{bp}(S)$, then conditions (a)--(c) in Lemma \ref{lem-bp-chain} hold, where we use Theorem \ref{thm-t-j-twist} (i) to check that $T_\beta \cap \Gamma$ is trivial for any curve $\beta$ of $v$.
The lemma is proved similarly to Lemma \ref{lem-c-stab}.
\end{proof}

Using Theorem \ref{thm-t-aut} and the above two lemmas, one can show Theorem \ref{thm-taut-t-j} (i) following the same line of proof as in Theorem \ref{thm-self}.
We leave the details to the reader.


\subsection{Geometric subgroupoids for the Johnson kernel}\label{subsec-j}

Throughout this subsection, we fix the following notation:
Let $S=S_{g, k}$ be a surface with either $g=1$ and $k\geq 3$; $g=2$ and $k\geq 2$; or $g\geq 3$ and $k\geq 0$.
Denote by $\Gamma =\calk(S)\cap \mod(S; 3)$ and let $\Gamma \ca (X, \mu)$ be a p.m.p.\ action.
Denote by $\calg =\Gamma \ltimes X$ and let $\rho \colon \calg \to \Gamma$ be the canonical projection.
For any set $V$ on which $\Gamma$ acts and for every element $v \in V$, we denote by $\Gamma_v$ the stabilizer of $v$ in $\Gamma$ and set $\calg_v =\Gamma_v \ltimes X$.
For every element $v \in V_s(S)$, we denote by $\calt_v =(T_v\cap \Gamma) \ltimes X$.

\begin{lem}\label{lem-2-cases-j}
Let $Y$ be a non-negligible subset of $X$, $\calm$ a nowhere amenable subgroupoid of $\calg|_Y$, and $\caln$ an amenable and nowhere finite subgroupoid with $\caln \vartriangleleft \calm$.
Suppose the following condition holds:
\begin{enumerate}
\item[($\ast$)] If $\calm_1$ and $\caln_1$ are subgroupoids of $\calg|_Y$ such that $\caln_1$ is amenable and nowhere finite and we have $\calm <\calm_1$ and $\caln_1\vartriangleleft \calm_1$, then $\calm =\calm_1$.
\end{enumerate} 
Then, there exists a partition $Y=\bigsqcup_nY_n$ into countably many measurable subsets such that for any $n$, one of the following cases occurs:
\begin{enumerate}
\item[(1)] There exists an $\alpha \in V_s(S)$ with $\calm|_{Y_n}=\calg_\alpha|_{Y_n}$.
\item[(2)] All of the CRS's and the T, IA and IN systems of $\calm|_{Y_n}$ and $\caln|_{Y_n}$ are constant, and letting $\sigma \in \Sigma(S)$ to be the value of the CRS of $\calm|_{Y_n}$, we have the following:
No curve in $\sigma$ is separating in $S$, and there exists an IA component of $\caln|_{Y_n}$.
\end{enumerate}
\end{lem}

\begin{proof}
The proof will be omitted as it follows verbatim as the proof of Lemma \ref{lem-2-cases} above. We only remark that in the course of the proof, one has to use Theorem \ref{thm-t-j-twist} (ii) in place of Lemma \ref{lem-twist}.
\end{proof}

\begin{lem}\label{lem-k-stab}
Pick $\alpha \in V_s(S)$.
Let $Y$ be a non-negligible subset of $X$ and set $\calm =\calg_\alpha |_Y$.
Then, the following assertions hold:
\begin{enumerate}
\item[(i)] The group $\Gamma_\alpha$ is non-amenable, and $\calm$ is therefore nowhere amenable.
\item[(ii)] The groupoid $\calm$ satisfies condition ($\ast$) in Lemma \ref{lem-2-cases-j}.
\item[(iii)] Let $\call$, $\calm_1$, $\calm_2$ and $\caln_1$ be subgroupoids of $\calg|_Y$ such that $\calm_2$ is nowhere amenable, $\caln_1$ is amenable and nowhere finite, and we have $\calm < \call$, $\caln_1\vartriangleleft \calm_1$, $\calm_2\vartriangleleft \calm_1$ and $\calm_2 \vartriangleleft \call$.
Then $\calm =\call$.
\item[(iv)] If $\caln$ is an amenable and nowhere finite subgroupoid of $\calm$ with $\caln \vartriangleleft \calm$, then there exists a partition $Y=\bigsqcup_nY_n$ into countably many measurable subsets such that $\caln|_{Y_n}<\calt_\alpha|_{Y_n}$ for any $n$.
\end{enumerate}
\end{lem}

\begin{proof}
For any component $Q$ of $S_\alpha$, either $Q$ is a pair of pants, $Q$ is a handle, or $\theta_Q(\Gamma_\alpha)$ is non-elementary.
Applying Lemma \ref{lem-s-chain} as in the proof of Lemma \ref{lem-i-stab}, we obtain the lemma.
\end{proof}

Using Theorem \ref{thm-cs-aut} and the above two lemmas, one can show Theorem \ref{thm-taut-t-j} (ii) following the same line of proof as in Theorem \ref{thm-self}.
Again, we leave the details to the reader.



\section{Applications to $ME$ and $OE$ superrigidity}\label{sec-me-oe}

In this section we summarize several applications of the tautness results proved in the Sections \ref{sec-sbg}-\ref{sec-t-j}.
Deriving measure equivalence ($ME$) and orbit equivalence ($OE$) rigidity results from the tautness property is a method which traces back to the influential work of Furman \cite{Fu99a, Fu99b, Fu01}. Since then, this strategy was successfully applied in many subsequent developments, e.g., \cite{BFS10, Ki06, Ki08, Ki09b, Ki10, MS04, Sa09}.
Recycling these methods we obtain several new $ME$ and $OE$ rigidity results for large classes of surface braid groups, Torelli groups, and Johnson kernels. For the proofs of these results we will refer the reader to the relevant previous papers. For instance, applying the method of deriving ME-rigidity from tautness shown in \cite[Theorem 3.5]{Ki09b} on the basis of \cite{Fu99a}, we obtain the following:

\begin{theorem}\label{thm-mer}
Let $S=S_{g, k}$ be a surface and denote by $G=\mod^*(S)$. Let $\Gamma$ be a group in one of the following classes:
\begin{itemize}
\item $\Gamma =P(S)$, where $g\geq 2$, $k\geq 2$. 
\item $\Gamma =\cali(S)$, where either $g=1$, $k\geq 3$; $g=2$, $k\geq 1$; or $g\geq 3$, $k\geq 0$.
\item $\Gamma =\calk(S)$, where either $g=1$, $k\geq 3$; $g=2$, $k\geq 2$; or $g\geq 3$, $k\geq 0$.
\end{itemize}
Then for every countable group $\Lambda$ that is measure equivalent to $\Gamma$ there exists a homomorphism $\phi \colon \Lambda \to G$ such that $\ker \phi$ is finite and $\phi(\Lambda)$ is commensurable with $\Gamma$.
\end{theorem}

Also, combining our tautness results with the results in \cite[Section 2]{Fu01} we derive the following:

\begin{theorem}\label{thm-lat}
Denote by $S$, $G$ and $\Gamma$ as in Theorem \ref{thm-mer}. Let $\Gamma_1$ be a finite index subgroup of $\Gamma$, $H$ be a second countable, locally compact group, and $\iota \colon \Gamma_1 \to H$ be an injective homomorphism such that $\iota(\Gamma_1)$ a lattice of $H$.
Then, there exists a continuous homomorphism $\Phi \colon H\to G$ such that $\ker \Phi$ is compact, the image $\Phi(H)$ is commensurable with $\Gamma$, and for any $\gamma \in \Gamma_1$, the equation $\Phi(\iota(\gamma))=\gamma$ holds.
In particular, $H_1=\Phi^{-1}(\Gamma_1)$ is a finite index subgroup of $H$ that admits the following semi-direct decomposition $H_1=\iota(\Gamma_1)\ltimes \ker \Phi$.
\end{theorem}

To state properly the orbit equivalence rigidity results we recall some definitions:

\begin{definition}
Two ergodic p.m.p.\ actions $\Gamma \ca (X, \mu)$ and $\Lambda \ca (Y, \nu)$ are called
\begin{enumerate}
\item \textit{conjugate} if there exist a probability space isomorphism $\psi\colon (X, \mu)\to (Y, \nu)$ and a group isomorphism $\delta \colon \Gamma \to \Lambda$ such that $\psi(\gamma x)=\delta(\gamma)\psi(x)$, for any $\g\in\Gamma$ and almost every $x\in X$.
\item \textit{stably conjugate} if there exist finite index subgroups $\Gamma_0<\Gamma$, $\Lambda_0<\Lambda$, and finite normal subgroups $N\vartriangleleft \Gamma_0$, $M\vartriangleleft \Lambda_0$ such that
\begin{itemize}
\item the action $\Gamma \ca (X, \mu)$ is induced from some p.m.p.\ action $\Gamma_0\ca (X_0, \mu_0)$.
\item the action $\Lambda\ca (Y, \nu)$ is induced from some p.m.p.\ action $\Lambda_0\ca (Y_0, \nu_0)$.
\item the actions $\Gamma_0/N\ca (X_0, \mu_0)/N$ and $\Lambda_0/M\ca (Y_0, \nu_0)/M$ are conjugate.
\end{itemize} 
\item \textit{orbit equivalent ($OE$)} if there exists a probability space isomorphism $\psi\colon (X, \mu)\to (Y, \nu)$ such that $\psi(\Gamma x)=\Lambda \psi(x)$, for almost every $x\in X$.
\item \textit{stably orbit equivalent (stably $OE$)} if there exist non-negligible subsets $A\subset X$, $B\subset Y$ and a probability space isomorphism $\psi\colon (A, \mu(A)^{-1}\mu|_A)\to (B, \nu(B)^{-1}\nu|_B)$ such that $\psi(\Gamma  x\cap A)=\Lambda \psi(x)\cap B$, for almost every $x\in A$.
\end{enumerate} 
Here, we say that an ergodic p.m.p.\ action $\Gamma \ca (X,\mu)$ is \textit{induced} from a p.m.p.\ action $\Gamma_0\ca (X_0, \mu_0)$ of a finite index subgroup $\Gamma_0$ of $\Gamma$ if $X_0$ is a $\Gamma_0$-invariant and non-negligible subset of $X$ such that $\mu(\gamma X_0\cap X_0)=0$ for any $\gamma \in \Gamma \setminus \Gamma_0$.

Finally, we say that a p.m.p.\ action $\Gamma \ca (X,\mu)$ is \textit{aperiodic} if any finite index subgroup of $\Gamma$ acts on $(X, \mu)$ ergodically.

It is clear from the definitions that (stable) conjugacy implies (stable) OE for actions. The reversed implications are false in general and whenever they hold are labeled as \emph{$OE$-rigidity} phenomena. A free, ergodic, p.m.p.\ action $\G \ca (X, \mu)$ is called
\begin{enumerate}
\item \emph{$OE$-superrigid} if whenever $\La \ca (Y, \nu)$ is a free, ergodic, p.m.p.\ action which is $OE$ to $\Gamma \ca (X, \mu)$, the actions $\Lambda \ca (Y, \nu)$ and $\Gamma \ca (X, \mu)$ are conjugate.
\item \emph{stably $OE$-superrigid} if whenever $\La \ca (Y, \nu)$ is a free, ergodic, p.m.p.\ action which is stably $OE$ to $\Gamma \ca (X, \mu)$, the actions $\Lambda \ca (Y, \nu)$ and $\Gamma \ca (X, \mu)$ are stably conjugate.
\end{enumerate}
\end{definition}

Our tautness results in combination with the technique from \cite{Fu99a, Fu99b} provide many new examples of $OE$-superrigid actions.  Also, as its by-product we obtain new examples of countable p.m.p.\ equivalence relations with trivial fundamental group.

\begin{theorem}\label{oesuperrigidity}
Let $\Gamma$ be the group as in Theorem \ref{thm-mer}. Then any free, ergodic, p.m.p.\ action $\G\ca (X,\mu)$ is stably $OE$-superrigid.
If the action is further aperiodic, then it is $OE$-superrigid.
\end{theorem}

\begin{theorem}\label{trivial}
Let $\Gamma$ be the group as in Theorem \ref{thm-mer}. Then the equivalence relation arising from any free, ergodic, p.m.p.\ action $\G\ca (X,\mu)$ has the trivial fundamental group.
\end{theorem}

In the remaining part of the section we show the tautness property for direct products of groups in Theorem \ref{thm-mer}.
Recall that Monod-Shalom's class $\mathcal{C}$ consists of countable groups $\Gamma$ which admit a mixing unitary representation $\pi$ on a Hilbert space such that the second bounded, $\pi$-valued cohomology group $H_b^2(\Gamma, \pi)$ does not vanish \cite{MS04}.

First we review a few general facts on tautness. Let $\Gamma$ be a countable group, and denote by $\comm(\Gamma)$ the abstract commensurator of $\Gamma$. There exists a natural homomorphism from $\Gamma$ into $\comm(\Gamma)$ which is injective if and only if $\Gamma$ is ICC, \cite[Lemma 3.8]{Ki09b}.
If there exists a countable group $G$ such that $\Gamma <G$ and $\Gamma$ is taut relative to $G$ then it follows from \cite[Lemma 3.9]{Ki09b} that $G$ contains $\comm(\Gamma)$ and $\Gamma$ is taut relative to $\comm(\Gamma)$.

\begin{theorem}
Suppose that $\Gamma_1,\ldots, \Gamma_n$ are countable ICC groups in Monod-Shalom's class $\mathcal{C}$ and that each $\Gamma_i$ is taut relative to $\comm(\Gamma_i)$.
Then the direct product $\Gamma_1\times \cdots \times \Gamma_n$ is taut relative to $\comm(\Gamma_1\times \cdots \times \Gamma_n)$.
\end{theorem}

\begin{proof}
The proof relies heavily upon the methods used \cite{MS04} and it is essentially given in \cite[Section 7]{Ki06}, where direct products of mapping class groups are treated.
Let  $\Gamma =\Gamma_1\times \cdots \times \Gamma_n$ and $G_i=\comm(\Gamma_i)$.
We set
\[I=\{ \, (i, j)\in \{ 1,\ldots, n\}^2\mid \Gamma_i\ \textrm{and}\ \Gamma_j\ \textrm{are\ commensurable.}\, \}.\]
Denote by $\eta_i^i\colon \Gamma_i\to \Gamma_i$ the identity map.
For $(i, j)\in I$ with $i<j$, we fix an isomorphism $\eta_i^j\colon \Gamma_i^0\to \Gamma_j^0$ between finite index subgroups $\Gamma_i^0<\Gamma_i$ and $\Gamma_j^0<\Gamma_j$.
For $(i, j)\in I$ with $i>j$, we define $\eta_i^j$ to be the inverse of $\eta_j^i$.
Let $T$ be the set of bijections $t\colon \{ 1,\ldots, n\} \to \{ 1,\ldots, n\}$ such that $(i, t(i))\in I$, for all $i$.

The group $G_1\times \cdots \times G_n$ is naturally a subgroup of $\comm(\Gamma)$.
For any bijection $t\in T$ there exists an element $\eta_t\in\comm(\Gamma)$ induced by the isomorphism
\[(\gamma_1,\ldots, \gamma_n)\mapsto (\eta_{t(1)}^1(\gamma_{t(1)}),\ldots, \eta_{t(n)}^n(\gamma_{t(n)})).\]
Let $G$ be the subgroup of $\comm(\Gamma)$ generated by $G_1\times \cdots \times G_n$ and the elements $\eta_t$.
The group $G$ is a finite extension of $G_1\times \cdots \times G_n$ and does not depend on the choice of $\eta_i^j$. Next we will show that $\Gamma$ is taut relative to $G$. Before doing so we briefly notice that in particular this implies the equality $G=\comm(\Gamma)$ which is not necessary for the proof of the theorem.

Let $(\Sigma, m)$ be a self-coupling of $\Gamma$. To distinguish the two actions of $\Gamma$, we set $\Lambda_i=\Gamma_i$ and $\Lambda =\Gamma$, and we view $(\Sigma, m)$ as a $(\Gamma, \Lambda)$-coupling.
We will identify $\Gamma_i$ with the subgroup $\{ e\} \times \cdots \times \{ e\} \times \Gamma_i\times \{ e\} \times \cdots \times \{ e\}$ of $\Gamma$, and $\Gamma$ with the subgroup $\Gamma \times \{ e\}$ of $\Gamma \times \Lambda$. Similarly, the subgroups $\Lambda_i$ and $\Lambda$ will be identified with subgroups of $\Gamma \times \Lambda$.
proceeding as in the proof of \cite[Proposition 5.1 and Theorem 1.16]{MS04}, there exists a decomposition $\Sigma =\bigsqcup_{t\in T}\Sigma_t$ into $(\Gamma \times \Lambda)$-invariant measurable subsets such that for each $t\in T$, there exist fundamental domains $X_t, Y_t\subset \Sigma_t$ for the actions $\Lambda \ca \Sigma_t$ and $\Gamma \ca \Sigma_t$, respectively, satisfying $\Gamma_iX_t\subset \Lambda_{t(i)}X_t$ and $\Lambda_{t(i)}Y_t\subset \Gamma_iY_t$, for all $i$.
We may assume that $\Sigma =\Sigma_t$ for some $t\in T$, and set $X=X_t$ and $Y=Y_t$.

We set $\Gamma_i'=\prod_{j\neq i}\Gamma_j$ and $\Lambda_i'=\prod_{j\neq i}\Lambda_j$.
Let $\overline{\Sigma}_i$ be the space of ergodic components for the action $\Gamma_i'\times \Lambda_{t(i)}'\ca (\Sigma, m)$, and denote by $\pi_i\colon \Sigma \to \overline{\Sigma}_i$ the projection.
We define a measure $\mu_i$ on $\overline{\Sigma}_i$ by projecting the restriction of $m$ to $\Gamma_iY$ and a measure $\nu_i$ on $\overline{\Sigma}_i$ by projecting the restriction of $m$ to $\Lambda_{t(i)}X$. Notice that $\mu_i$ and $\nu_i$ are both $(\Gamma_i\times \Lambda_{t(i)})$-invariant and are equivalent.
It follows that there exists a measure $m_i$ equivalent to $\mu_i$ and with respect to which $\overline{\Sigma}_i$ is a $(\Gamma_i, \Lambda_{t(i)})$-coupling.
For a precise proof, the reader is referred to the proof of \cite[Theorem 1.16]{MS04}.

The space $(\overline{\Sigma}_i, m_i)$ is a self-coupling of $\Gamma_i$ given by the formula $(\gamma, \lambda)x=(\gamma, \eta_i^{t(i)}(\lambda))x$, for $\gamma, \lambda \in \Gamma_i$ and $x\in \overline{\Sigma}_i$.
By the tautness assumption, there exists a measurable, $(\Gamma_i\times \Gamma_i)$-equivariant map $\overline{\Phi}_i\colon \overline{\Sigma}_i\to G_i$. Denote by $\Phi_i=\overline{\Phi}_i\circ \pi_i\colon \Sigma \to G_i$, $s=t^{-1}$, and define a map $\Phi \colon \Sigma \to G$ by letting
\begin{align*}
\Phi((\lambda_1,\ldots, \lambda_n)x)&=(\Phi_1(x),\ldots, \Phi_n(x))(\eta_{s(1)}^1(\lambda_{s(1)}),\ldots, \eta_{s(n)}^n(\lambda_{s(n)}))^{-1}\eta_s\\
&=(\Phi_1(x),\ldots, \Phi_n(x))\eta_s(\lambda_1,\ldots, \lambda_n)^{-1}
\end{align*}
for all $\lambda_i \in \Lambda_i$ and $x\in X$.
Finally, one can check that this is well-defined and $(\Gamma \times \Lambda)$-equivariant.
\end{proof}

By \cite[Corollary B]{Ha08}, all groups $\mod(S_{g, k})$, with $3g+k-4>0$ belong to $\mathcal{C}$.
Moreover, by \cite[Proposition 7.4]{MS04}, the class $\mathcal{C}$ is also closed under taking a normal subgroup. Thus, all groups described in Theorem \ref{thm-mer} belong to $\mathcal{C}$ as well.
Consequently, we have the following:

\begin{cor}
Let $\Gamma_1,\ldots, \Gamma_n$ be any groups as in Theorem \ref{thm-mer}.
Then $\Gamma_1\times \cdots \times \Gamma_n$ is taut relative to $\comm(\Gamma_1\times \cdots \times \Gamma_n)$.
\end{cor}

\begin{Remark}\label{oerig-prod}As mentioned earlier, this tautness property implies that Theorems \ref{thm-mer}, \ref{thm-lat}, \ref{oesuperrigidity} and \ref{trivial} also hold for products groups $\G=\Gamma_1\times \cdots \times \Gamma_n$, where $\G_i$ are any groups as in Theorem \ref{thm-mer}.
\end{Remark}


\section{Applications to $W^*$-superrigidity}\label{sec-w}

In this section we will explain how the orbit equivalence superrigidity results obtained in the previous section can be used in combination with the uniqueness of Cartan subalgebra results proved in \cite{CIK13} to produce new examples of (stably) $W^*$-superrigid actions. This will include large classes of actions by many of the normal subgroups of mapping class groups considered in the previous sections. To be able to state these results we will first recall some terminology and will provide some context from \cite{CIK13}. 

\begin{definition}  Denote by $\mathcal C_{rss}$ the collection of groups consisting of all non-elementary hyperbolic groups and all non-amenable, non-trivial free products of groups.  By definition we let  ${\bf Quot_{1}}(\mathcal C_{rss}):=\mathcal C_{rss}$. For a given integer $n\geq 2$, we denote by ${\bf Quot_n}(\Cal C_{rss})$ the collection of all groups $\G$ satisfying the following properties: 
\begin{itemize}
\item  there exists a family of groups  $\Gamma_k$, $1\leq k\leq n$, such that $\G=\G_n$, $\G_1\in \Cal C_{rss}$, and 
\item  there exists a family of surjective homomorphisms $ \pi_k\colon \G_{k}\rightarrow \G_{k-1}$ such that $\ker(\pi_k)\in \Cal C_{rss}$, for all $2\leq k\leq n$.
\end{itemize}
We denote by ${\bf Quot}(\mathcal C_{rss}):=\bigcup_{n\in\mathbb N} {\bf Quot_n}(\mathcal C_{rss})$, the class of all \emph{finite-step extensions of groups in $\mathcal C_{rss}$}.
\end{definition}

Using recent striking classification results for normalizers of amenable subalgebras in finite von Neumann algebras due to Popa and Vaes \cite{PV11,PV12} and to Ioana \cite{Io12}, the following dichotomy was proved in \cite[Theorem 3.1]{CIK13}: Let $\G \in {\bf Quot}(\mathcal C_{rss})$, let $\G\ca (B,\tau)$ be a trace-preserving action on a finite von Neumann algebra $(B,\tau)$ and denote by $P=B\rtimes \G$ the corresponding crossed product von Neumann algebra. Then for any masa $A\subset P$ one of the following holds: 1) a corner of $A$ can be intertwined into $B$ inside $P$ in the sense of Popa \cite[Theorem 2.1, Corollary 2.3]{Po03} or  2) the normalizing algebra  $\Cal N_{P}(A)''$ has infinite Pimsner-Popa index inside $P$ \cite[Theorem 2.2]{PP86}. For similar results regarding actions by group extensions the reader may also consult S. Vaes and  P. Verraedt's more recent work \cite{VV14}. In particular, our dichotomy can be used to produce many new examples of free, ergodic, p.m.p.\ group actions on probability spaces which give rise to von Neumann algebras with unique Cartan subalgebra.

\begin{theorem}\label{cartan1}\cite[Corollary 3.2]{CIK13} If $\G\in {\bf Quot}(\mathcal C_{rss})$ then any free, ergodic, p.m.p.\ action $\G \ca (X,\mu)$ is $\mathcal C$-superrigid.  
\end{theorem}

The class ${\bf Quot}(\mathcal C_{rss})$ is fairly large and includes many natural families of groups intensively studied in other areas of mathematics, such as all non-elementary groups that are hyperbolic relative to finite families of residually finite, infinite, proper subgroups \cite{Os06,DGO11};  all mapping class groups associated to punctured surfaces $S_{g,k}$ where either $g=0$, $k\geq 4$; $g=1$, $k\geq 1$; or $ g=2$, $k\geq 0$.  For a more comprehensive list of groups in ${\bf Quot}(\mathcal C_{rss})$ we refer the reader to \cite[Section 4.3]{CIK13} and \cite[Section 3]{CKP14} together with all references therein. In particular, \cite[Theorem 3.7 and Theorem 3.9]{CKP14} show that most surface braid groups along with most Torelli groups and Johnson kernels associated with surfaces of low genus are natural examples of groups in ${\bf Quot}(\mathcal C_{rss})$; moreover, by \cite[Lemma 2.10]{CIK13} the same holds for any direct product of finitely many such groups.  Hence, Theorem \ref{cartan1} is applicable to all these groups and by further combining it with Theorem \ref{oesuperrigidity} and Remark \ref{oerig-prod} we obtain new examples of $W^*$-superrigid actions. These add to the previous examples found in \cite{Pe09,PV09,FV10,CP10,HPV10,Io10,CS11,CSU11,PV11,PV12,Bo12,CIK13}.

\begin{cor}\label{wsuperrigid} Let $\G=\G_1\times \cdots \times \G_n$ be a product group such that each factor $\G_i$ is in one of the following classes:
\begin{itemize}
\item [(i)] $P(S_{g,k})$, where $g\geq 2$, $k\geq 2$.
\item [(ii)] $\Cal I (S_{g,k})$, where either $g =1$, $k\geq 3$; or $g =2$, $k\geq 1$.
\item [(iii)] $\Cal K (S_{g,k})$ where either $g= 1$, $k\geq 3$; or $g =2$, $k\geq 2$.
\end{itemize}
Then any free, ergodic, p.m.p.\ action $\G\ca (X,\mu)$ is stably $W^*$-superrigid. If the action is further aperiodic, then it is $W^*$-superrigid. 
\end{cor}

Since by Remark \ref{identification} the central quotient $P(S_{1,k})$ coincides with $\Cal I (S_{1,k})$ for all $k\geq 3$ the parts (i) and (ii) in the previous corollary together with   \cite[Theorem A]{CIK13} settle completely the $W^*$-superrigidity question for all free, aperiodic actions by (direct products of finitely many) central quotients of most surface braid groups. 

We conjecture that this $W^*$-superrigidity result also holds for the remaining cases, in particular, for all free, ergodic, p.m.p.\ actions of Torelli groups and Johnson kernels $\G= \Cal I (S_{g,k}),\Cal K (S_{g,k})$ with $g\geq 3$ and $k\geq 0$. While we have already seen in the previous section that the OE-superrigidity holds, establishing the $\Cal C$-superrigidity property seems a more challenging problem, even for particular actions such as profinite or Bernoulli, as it requires a completely new technological approach.

Finally, we point out that the previous results can be exploited to produce new examples of II$_1$ factors with trivial fundamental group.

\begin{cor}If $\G$ is the group as in Corollary \ref{wsuperrigid} and $\G \ca(X, \mu)$ is a free, ergodic, p.m.p.\ action then the corresponding II$_1$ factor $P = L^\infty(X)\rtimes \Gamma$ has the trivial fundamental group.
\end{cor}
\begin{proof} Let $\mathcal R = \{ \, (x, y) \in X \times X\mid \G x = \G y\, \}$ be the equivalence relation arising from the action $\G \ca (X, \mu)$.
By Theorem \ref{cartan1} $L^\infty(X)$ is the unique Cartan subalgebra of $P$, and hence the fundamental groups of $P$ and $\mathcal{R}$ are isomorphic. By Theorem \ref{trivial} the conclusion of the corollary follows.
\end{proof}

\section{$\mathcal C$-superrigidity for actions by residually hyperbolic groups}

In this section we use techniques similar with the ones developed in \cite{CIK13} to show that free, mixing, p.m.p.\ actions by many finitely presented, residually free groups are $\mathcal C$-superrigid (see Theorem \ref{c-sr1} and Corollary \ref{c-sr} below). Before proceeding to the proofs of these results  we record three elementary group theory propositions which will be used later. We thank the referee for simplifying our original proof of the following proposition.

 \begin{prop}\label{abelianreduction}Let $\G$, $H$ and $A$ be countable groups such that $\G$ is ICC and $A$ is virtually abelian. Assume there exists $\Phi \colon \G \ra H\times A$ an injective group homomorphism such that $\pi_A\circ \Phi(\G)=A$; here $\pi_A\colon H\times A\ra A$ is the canonical projection. Then there exists a finite index subgroup $\G_0< \G$ and an injective group homomorphism $\rho\colon\G_0\ra H$.
 \end{prop}
 
\begin{proof} Notice that the ICC condition is preserved under passing to finite index subgroups. Hence, by passing to a finite index subgroup of $\G$, we can assume that $A$ is abelian. 
Denote by $\rho =\pi_H \circ \Phi$ the composition homomorphism $\G \ra H\times A\ra H$.
Let $Z$ be the kernel of $\rho$.
Then $\Phi(Z)\subset \{ e\} \times A$.
Since $\Phi$ is injective and $\{ e\}\times A$ is in the center of $H\times A$ it follows that $Z$ is in the center of $\G$.
But $\G$ is ICC, so its center is trivial.
This implies $Z=\{ e\}$ and hence $\rho$ is the desired injective homomorphism.  
\end{proof}

A subgroup $\Sigma <\G$ is called \emph{malnormal} in $\G$ if for any $\g\in \G\setminus \Sigma$ we have $\g\Sigma \g^{-1} \cap \Sigma=\{ e\}$. 
 
\begin{prop} \label{malnorm}Let $\theta\colon\G\ra \La$ be a group homomorphism and let $L <\La$ be a malnormal subgroup. Then for every $h,k\in \G\setminus \theta^{-1}(L)$ one of the following holds: \begin{enumerate}
\item $h\theta^{-1}(L) k^{-1}\cap \theta^{-1}(L)=\emptyset$;
\item there exists $\omega\in \theta^{-1}(L)$ such that $h\theta^{-1}(L) k^{-1}\cap \theta^{-1}(L)=\omega \ker(\theta)$.
\end{enumerate} 
 \end{prop}

\begin{proof} Assume (1) does not hold; there exists $\omega\in h\theta^{-1}(L) k^{-1}\cap \theta^{-1}(L)$. Thus one can find $\sigma\in \theta^{-1}(L)$ so that $h \sigma k^{-1}=\omega$ and hence $h=\omega k\sigma^{-1}$. Using this we get $h\theta^{-1}(L) k^{-1}\cap \theta^{-1}(L)=\omega k\sigma^{-1}\theta^{-1}(L) k^{-1}\cap \theta^{-1}(L)=\omega (k\theta^{-1}(L) k^{-1}\cap \theta^{-1}(L))$. 

On the other hand since $k\in \G\setminus \theta^{-1}(L)$ we have $\ker(\theta)< k \theta^{-1}(L)k^{-1}\cap \theta^{-1}(L)< \theta^{-1}(\theta(k)L\theta(k)^{-1}\cap L)$ and using $\theta(k)L\theta(k)^{-1}\cap L=\{ e\}$ we obtain that $k\theta^{-1}(L) k^{-1}\cap \theta^{-1}(L)=\ker(\theta)$. This, together with the previous paragraph give (2). \end{proof}

The next result classifies amenable subgroups of HNN extensions and amalgamated free products over malnormal subgroups.

\begin{lem}\label{lem-tree}
The following assertions hold:
\begin{enumerate}
\item[(i)] Let $G=\La \star_A B$ be an amalgamated free product such that $A$ is proper in $\La$ and $B$, and is malnormal in $\La$.
Then any amenable subgroup of $G$ is either infinite cyclic or contained in a conjugate of $\La$ or $B$ in $G$.
\item[(ii)] Let $\La$ be a group, $A$ a malnormal subgroup of $\La$, and $\phi \colon A\to \La$ an injective homomorphism.
Suppose that $\la A\la^{-1}\cap \phi(A)$ is trivial for any $\la \in \La$.
We define $G=\langle \, \La, t\mid tat^{-1}=\phi(a),\ a\in A \, \rangle$ as the HNN extension.
Then any amenable subgroup of $G$ is either infinite cyclic or contained in a conjugate of $\La$ in $G$.
\end{enumerate}
\end{lem}

\begin{proof}
We prove assertion (i).
Let $T$ denote the Bass-Serre tree associated with the decomposition $G=\La \star_A B$.
Since $A$ is malnormal in $\La$, no non-neutral element of $G$ fixes a geodesic segment in $T$ of length more than $2$.
Let us refer this fact as $(\star)$.

Let $g$ be a non-neutral element of $G$.
We say that $g$ is \textit{elliptic} if $g$ fixes some vertex of $T$.
We say that $g$ is \textit{hyperbolic} if $g$ fixes a bi-infinite geodesic in $T$, called the \textit{axis} of $g$, and acts on it by translation.
It follows from \cite[Proposition I.4.11]{DD89} that any non-neutral element of $G$ is either elliptic or hyperbolic.
Let $\partial T$ be the boundary of $T$.
For a hyperbolic element $g\in G$, let $g^+$ and $g^-$ denote the two fixed points of $g$ in $\partial T$ such that for any point $x$ of $T$, we have $g^nx\to g^+$ and $g^{-n}x\to g^-$ as $n\to +\infty$.

\begin{claim}\label{claim-tree}
Let $g, h\in G$ be hyperbolic elements.
Suppose that there exists a point $x\in \partial T$ fixed by $g$ and $h$.
Then $g^p=h^q$ for some non-zero integers $p$, $q$, and thus $\{ g^\pm \}=\{ h^\pm \}$.
\end{claim}

\noindent \emph{Proof of Claim \ref{claim-tree}.}
We may assume $x=g^+=h^+$ by exchanging $h$ with $h^{-1}$ if necessary.
There exists a geodesic ray $l$ starting at some vertex of $T$ and toward $x$ such that $gl\subset l$ and $hl\subset l$.
We then have positive integers $p$, $q$ such that $g^{-p}h^q$ fixes any vertex in $l$.
By fact $(\star)$, the equality $g^p=g^q$ holds. $\hfill\blacksquare$
\vskip 0.05in

Let $M$ be an amenable subgroup of $G$.
If any non-neutral element of $M$ is elliptic, then by \cite[Theorem I.4.12]{DD89} and fact $(\star)$, the group $M$ fixes a vertex of $T$, and is hence contained in a conjugate of $\La$ or $B$.

Suppose otherwise.
We have a hyperbolic element $g\in M$.
We show that any element $h\in M$ fixes $g^+$ and $g^-$.
Assuming $hg^+\not \in \{ g^\pm \}$, we deduce a contradiction.
If $hg^-\in \{ g^\pm \}$, then applying Claim \ref{claim-tree} to $hgh^{-1}$ and $g$, we have $\{ hg^\pm \} =\{ g^\pm \}$. 
This contradicts our assumption.
We therefore have $hg^-\not \in \{ g^\pm \}$.
By the ping-pong argument, we can show that the group generated by $g$ and $hgh^{-1}$ contains $\mathbb{F}_2$.
This contradicts that $M$ is amenable.
It follows that $hg^+\in \{ g^\pm \}$.
By Claim \ref{claim-tree}, we have $\{ hg^\pm \} =\{ g^\pm\}$.
Since any element of $G$ fixing an edge of $T$ also fixes its two vertices, we have $hg^\pm =g^\pm$.

We have shown that any element of $M$ fixes the axis of $g$.
By measuring the translation distance, we obtain a homomorphism $\delta \colon M\to \mathbb{Z}$. 
The kernel of $\delta$ fixes any vertex of the axis, and is therefore trivial by fact $(\star)$.
It follows that $M$ is isomorphic to $\mathbb{Z}$.
Assertion (i) was proved.

Assertion (ii) can also be proved along similar argument.
The assumption implies that in the associated Bass-Serre tree, for any geodesic segment of length more than $2$, its stabilizer in $G$ is trivial.
\end{proof}

\subsection{Popa's intertwining-by-bimodules technique} Popa introduced in  \cite [Theorem 2.1 and Corollary 2.3]{Po03} a powerful  technical tool to find intertwiners between subalgebras of a tracial von Neumann algebra called the {\it intertwining-by-bimodules technique}. On a von Neumann algebra $(M,\tau)$ endowed with a faithful, normal, tracial state $\tau$ we have a Hilbert norm given by $\|x\|_2=\tau(x^*x)^{1/2}$, for all $x\in M$.

\begin {theorem} \cite [Theorem 2.1 and Corollary 2.3]{Po03}\label{corner} Let $(M,\tau)$ be a separable tracial von Neumann algebra and let $P,Q\subset M$ be (not necessarily unital) von Neumann subalgebras. Then the following are equivalent:

\begin{enumerate}

\item There exist  non-zero projections $p\in P$, $q\in Q$, a $*$-homomorphism $\phi\colon pPp\rightarrow qQq$  and a non-zero partial isometry $v\in qMp$ such that $\phi(x)v=vx$, for all $x\in pPp$.

\item There is no sequence $u_n\in\mathcal U(P)$ satisfying $\|E_Q(xu_ny)\|_2\rightarrow 0$, for all $x,y\in M$.
\end{enumerate}
\end{theorem}

If one of the equivalent conditions of Theorem \ref{corner} holds true,  then we say that {\it a corner of $P$ embeds into $Q$ inside $M$} and write $P\preceq_{M}Q$. If $Pp'\preceq_{M}Q$, for any non-zero projection $p'\in P'\cap 1_PM1_P$, then we write $P\preceq^{s}_{M}Q$.

Now we establish a several facts which are essential in the proof of the main result. The first result provides a precise ``location'' of normalizers of various subalgebras in crossed products arising from actions by malnormal subgroups. The result is largely inspired  from \cite[Theorem 6.1]{Po81} and \cite[Theorem 3.1]{Po03} and a proof is included only for the sake of completeness.
  
\begin{prop}\label{int}Let $\theta\colon\G\ra \La$ be a group homomorphism and let $L <\La$ be a malnormal subgroup. Let $\G\ca A$ be a trace preserving action on a finite von Neumann algebra, consider the crossed products $M=A\rtimes \G$, $N=A\rtimes \theta^{-1}(\La)$, $Q=A\rtimes \ker(\theta)$ and notice that $Q\subset N\subset M$. Let $p\in N$ be a projection and let $R\subset pNp$ be a von Neumann subalgebra such that $R\npreceq_N Q$. Then $\mathcal N_{pMp}(R)''\subset pNp$.\end{prop}

\begin{proof} To get our conclusion it suffices to show that $\mathcal N_{pMp}(R)\subset pNp$. Fix $u\in \mathcal N_{pMp}(R)$ and let $\theta\colon R\ra R$ be a $*$-isomorphism such that $ux=\theta(x)u$, for all $x\in R$. Since $R\subset pNp$, this further implies that for all $x\in R$ we have \begin{equation}\label{2001}
vx=\theta(x)v, 
\end{equation}
where $v=u-E_N(u)$.

Since $R\npreceq_N Q$, by Theorem \ref{corner} one can find a sequence of unitaries $(x_n)_n\subset \mathcal U(R)$ such that \begin{equation}\label{2002}
\lim_{n\ra\infty}\|E_Q(ax_nb)\|_2=0, \text{ for all }a,b\in N.
\end{equation}
Fix $\ve>0$. By the Kaplansky Density Theorem there exist $v_\ve\in M$ and a finite subset $F_\ve\subset \G\setminus \theta^{-1}(L)$ such that $\|v_\ve\|_\infty\leq 2$, $\|v-v_\ve\|_2\leq \ve$, and $v_\ve$ belongs to the linear span of $\{ \, a_\g u_\g \mid a_\g\in A, \g\in F_\ve \, \}$. These estimates combined with (\ref{2001}) and Cauchy-Schwarz inequality lead to  
\begin{equation}\label{2003}
\begin{split}
\|v\|^2_2&=\tau( \theta(x^*_n) v x_n v^*)\leq 4\ve + \tau( \theta(x^*_n) v_\ve x_n v^*_\ve)\\
& \leq 4\ve +\|E_N(v_\ve x_n v^*_\ve)\|_2\leq 4\ve +4\sum_{\g,\la \in F_\ve} \|E_N(u_\g x_n u_{\la^{-1}})\|_2.
\end{split}
\end{equation}

Using the Fourier expansion of $x_n$ and Proposition \ref{malnorm}, basic calculations show that for every $\g,\la \in F_\ve$ there exists $\omega \in \theta^{-1}(L)$ such that \begin{equation*}\begin{split}
|E_N(u_\g x_n u_{\la^{-1}})\|^2_2&=\sum_{g\in \theta^{-1}(L)\cap \g^{-1} \theta^{-1}(L)\la} \|E_A(x_n u_{g^{-1}})\|^2_2 \\& =\sum_{g\in \omega \ker(\theta)} \|E_A(x_n u_{g^{-1}})\|^2_2=\|E_Q(u_{\omega^{-1}}x_n)\|^2_2.\end{split}
\end{equation*}
This further implies the existence of a finite subset $K_\ve\subset \theta^{-1}(L)$ and a positive integer $C_\ve\geq 1$ such that $\sum_{\g,\la \in F_\ve} \|E_N(u_\g x_n u_{\la^{-1}})\|_2\leq C_\ve \sum_{\omega \in K_\ve}\|E_Q(u_{\omega^{-1}}x_n)\|_2$. Notice that $C_\ve$ can be taken the highest number of repetitions of the same $\omega$ for different $\g$ and $\la$ as defined in the previous paragraph. Combining the previous inequality with (\ref{2003}) we obtain $\|v\|^2_2\leq 4\ve+4C_\ve \sum_{\omega \in K_\ve}\|E_Q(u_{\omega^{-1}}x_n)\|_2$. Since $K_\ve$ is finite then taking the limit over $n$ and using (\ref{2002}) we get $\|v\|^2_2\leq 4\ve$. Since $\ve>0$ was arbitrary we conclude that $v=0$ and hence $u=E_N(u)\in pNp$.\end{proof}

\begin{notation}\label{comu} Assume that $\G$ and $\La$ are countable groups and let $\delta\colon\G\ra \La$ be a group homomorphism. Let $\G\ca^\sigma (A,\tau)$ be a trace preserving action on a tracial von Neumann algebra $(A,\tau)$ and denote by $M=A\rtimes\G$ the corresponding crossed product von Neumann algebra.  We denote by $\{u_\g\}_{\g\in\Gamma}\subset L\Gamma$ and $\{v_\la\}_{\la\in\Lambda}\subset L\Lambda$ the canonical unitaries.
Consider the $*$-homomorphism $\Delta\colon M\rightarrow M\bar{\otimes}L\Lambda$ defined by 
 \begin{equation*}
\Delta(au_\g)=au_\g\otimes v_{\delta(\g)}\quad \text{for all $a\in A$ and $\g\in\Gamma$}.
\end{equation*} 
\end{notation}

Next we show an intertwining result for subalgebras of the form $\Delta(P)$ where $P\subset M$ is a subalgebra. Many of these are straightforward generalizations  of results from \cite{CIK13} and some proofs will be included only for the reader's convenience. 

\begin{prop} \label{intpreim}Assume we are in the setting from Notation \ref{comu}. Let $p\in M$ be a projection, let $P\subset pMp$ be a von Neumann subalgebra,  and let $\Sigma < \Lambda$ be a subgroup. 
If $\Delta(P)\preceq_{M\bar{\otimes}L\Lambda}M\bar{\otimes}L\Sigma$, then there exists $\la \in \La$ such that $P\preceq_{M}A\rtimes\delta^{-1}(\la \Sigma\la^{-1})$.\end{prop} 

\begin{proof} We proceed by contradiction; so assume that $P\npreceq_{M}A\rtimes\delta^{-1}(\la \Sigma\la^{-1})$, for all $\la \in \La$. Thus there exists a sequence of unitaries $(x_n)_n\subset \mathcal U(P)$ such that for all $x,y\in M$, $\la \in \La$ we have   
\begin{equation}\label{1}\lim_{n\ra \infty}\|E_{A\rtimes\delta^{-1}(\la \Sigma\la^{-1})}(xx_n y)\|_2=0.
\end{equation}
In the remaining part we will show that (\ref{1}) implies that for all $z,t\in M\bar \otimes L\La$ we have    
\begin{equation}\label{2}
\lim_{n\ra \infty}\|E_{M\bar \otimes L\Sigma}(z\Delta(x_n) t)\|_2=0,
\end{equation}
which, by Theorem \ref{corner}, further gives that $\Delta(P)\npreceq_{M\bar{\otimes}L\Lambda}M\bar{\otimes}L\Sigma$, a contradiction.

Using basic approximations it suffices to show (\ref{2}) only for elements of the form $z=au_{\g_1}\otimes v_{\la_1}$ and $t= b u_{\g_2} \otimes v_{\la_2} $, where $a,b \in A$, $\g_1,\g_2\in \G$, and $\la_1,\la_2\in\Lambda$.

If we consider the Fourier  expansion $x_n= \sum_\g x^n_\g u_\g$ where $a^n_\g\in A$, for all $\g\in \G$ and $n\in \mathbb N$ then a basic calculation shows  that 
\begin{equation}\label{3}\begin{split}
\|E_{M\bar \otimes L\Sigma}(z\Delta(x_n) t)\|^2_2 & =\|E_{M\bar \otimes L\Sigma}\left( \sum_{\g\in \G} au_{\g_1} x^n_\g bu_{\g_2} \otimes v_{\la_1} v_{\delta(\g)}v_{\la_2}\right) \|^2_2\\
&= \sum_{\substack{\g \in \G \\\la_1\delta(\g)\la_2\in \Sigma}   }\| a\sigma_{\g_1}(x^n_\g b) \|^2_2=\sum_{\substack{\g \in \G \\\la_1\delta(\g)\la_2\in \Sigma}   }\|\sigma_{\g^{-1}_1} (a) x^n_\g b \|^2_2.
\end{split}\end{equation} 

Let $F=\{ \, \g\in \G \mid \la_1 \delta(\g)\la_2\in \Sigma \, \}$. If $F=\emptyset$ then we already get (\ref{2}), so assume there exists $\g_0 \in F$. A straightforward calculation then shows that for every $\g\in F$ we have $\la_1 \delta(\g\g^{-1}_0)\la^{-1}_1=\la_1 \delta(\g)\delta(\g^{-1}_0)\la^{-1}_1 =\la_1 \delta(\g)\la_2 \la^{-1}_2(\delta(\g_0))^{-1}\la^{-1}_1$ $ =\la_1 \delta(\g)\la_2(\la_1 \delta(\g_0)\la_2)^{-1}\in \Sigma$, and hence $\g \in \delta^{-1}(\la_1 \Sigma \la^{-1}_1)\g_0$. In conclusion, $F\subset \delta^{-1}(\la_1 \Sigma \la^{-1}_1)\g_0$ and the last quantity in (\ref{3}) is smaller than
\begin{equation*}\label{5}\begin{split}
& \leq \sum_{\g \in \delta^{-1}(\la_1 \Sigma \la^{-1}_1)\g_0 }\|\sigma_{\g^{-1}_1} (a) x^n_\g b \|^2_2\\
&\leq \|a\|_\infty \|b\|_\infty \sum_{\g \in \delta^{-1}(\la_1 \Sigma \la^{-1}_1)\g_0 }\| x^n_{\g}\|^2_2\\
&\leq \|a\|_\infty \|b\|_\infty \|E_{A\rtimes \delta^{-1}(\la_1 \Sigma \la^{-1}_1)}( x_n u_{\g^{-1}_0})\|^2_2
\end{split}\end{equation*}     

Letting $x=1$ and $y= u_{\g^{-1}_0}$ in (\ref{1}) we see that the last quantity above converges to $0$ as $n\ra \infty$ and as a consequence we get (\ref{2}).
\end{proof}

\begin{definition}\label{ramgr} Let $C$ be a countable group and let $A,B< C$ be subgroups. Consider the quasi-regular representation $\lambda_B \colon C\ra \mathcal U(\ell^2(C/B))$. We say that \emph{$A$ is amenable relative to $B$ inside $C$} if there exists a sequence of unit vectors $(\xi_n)_n \subset \ell^2(C/B)$ such that $\lim_{n\ra\infty} \|(\la_B)_a(\xi_n)-\xi_n\|_2=0$ for all $a\in A$. Here $\|\cdot\|_2$ is the $2$-norm of the Hilbert space $\ell^2(C/B)$.
\end{definition}

For further reference we also notice the following proposition whose proof will be omitted since it is very similar to the proof of \cite[Proposition 3.5]{CIK13}.

\begin{prop}\label{rel-am}Assume we are in the setting from Notation \ref{comu}. Let $p\in M$ be a projection  and let $\Sigma < \Lambda$ be a subgroup. 
If $\Delta(pMp)$ is amenable relative to $M\bar{\otimes}L\Sigma$ inside  $M\bar{\otimes}L\Lambda$, in the sense of \cite[Definition 2.2]{OP07}, then the group $\delta(\G)$ is amenable relative to $\Sigma$ inside $\La$ as in Definition \ref{ramgr}. \end{prop}

A group $H$ is called \textit{CSA} if any maximal abelian subgroup of $H$ is malnormal in $H$ (\cite[Section 5]{MR96}). To properly state our main result we need to introduce a notation.


\begin{notation}\label{not2}Let $\G$ be a group, $\La$ be a non-elementary hyperbolic group and $H_1,\ldots, H_m$ be CSA groups. Assume for every $1\leq k\leq m$ the following exist:
\begin{enumerate}
\item[a)] A finite sequence of groups, $\La =\La^k_0,\ \La^k_1, \ldots , \La^k_{n_k}=H_k$, such that for every $i=1,\ldots, n_k$, the group $\La^k_i$ is either
\begin{itemize}
\item an amalgamated free product  $\La^k_{i-1} \star_{A^k_i} B^k_i$, where we assume that $B^k_i$ is an abelian group, and that $A^k_i$ is a common proper subgroup of  $\La^k_{i-1}$ and $B^k_i$ and is malnormal in $\La^k_{i-1}$, or
\item an HNN extension $\La^k_{i-1} \star_{A^k_i}=\langle \, \La^k_{i-1}, t\mid tat^{-1}=\phi^k_i(a),\ a\in A^k_i\, \rangle$, where we assume that $A^k_i$ is an abelian, malnormal subgroup of $\La^k_{i-1}$, that $\phi^k_i \colon A^k_i\to \La^k_{i-1}$ is an injective homomorphism, and that $\la A^k_i\la^{-1}\cap \phi^k_i(A^k_i)$ is trivial for any $\la \in \La^k_{i-1}$;
\end{itemize}
\item[b)] A group homomorphism  $\delta_k\colon \G \ra H_k$ such that $\bigcap^m_{l=1} \ker(\delta_l)=\{ e\}$. 
\end{enumerate}\end{notation}
\begin{Remark}\label{vab}
For any $i=1,\ldots, n_k$, Lemma \ref{lem-tree} is applicable to $\La^k_i=\La^k_{i-1} \star_{A^k_i} B^k_i$ or $\La^k_i=\La^k_{i-1} \star_{A^k_i}$. Notice that any amenable subgroup of $\La$ is virtually abelian.
By induction on $i$, it follows that any amenable subgroup of $H_k$ is virtually abelian.
\end{Remark}
  
\begin{lem}\label{dicho} Let $\G$ be an ICC group as in Notation \ref{not2}. Assume that $\G\ca (X,\mu)$ is  free, measure preserving action on a non-atomic probability space whose restriction to any non-amenable subgroup of $\G$ is ergodic. Denote by $M=L^{\infty}(X)\rtimes \G$ the corresponding group measure space von Neumann algebra and let $p\in M$ be a non-zero projection. Let $C\subset pMp$ be a masa such that $\mathcal N_{pMp}(C)''\subset pMp$ has finite index. Then for each $1\leq k\leq m$ one of the following holds:
\begin{enumerate}
\item [c)] $C\preceq^s_M L^\infty(X)\rtimes \ker(\delta_k)$;
\item [d)] $\delta_k(\G)$ is virtually abelian.
\end{enumerate} 
\end{lem}

\begin{proof}Throughout the proof we denote by $A=L^{\infty}(X)$. Fixing $k$ we will proceed by induction on $n_k$. When  $n_k=0$ we have that $\La=\La_0 $ is a non-elementary hyperbolic group. Let $\Delta \colon M\ra M\bar\otimes L\La$ be the von Neumann algebra $*$-homomorphism arising from $\delta_k$ as described in Notation \ref{comu}. Then, applying \cite[Theorem 3.1]{PV12} for subalgebra $\Delta(C)\subset M\bar \otimes L\La$, one of the following must hold:
\begin{enumerate}
\item[(1)] $\Delta(C)\preceq_{M\bar\otimes L\La} M\otimes 1$;
\item[(2)] $\Delta(P)$ is amenable relative to $M\otimes  1$ inside $M\bar\otimes L\La$.
\end{enumerate} 

Assuming situation (1), then Proposition \ref{intpreim} implies that $C\preceq_M A\rtimes \ker(\delta_k)$. Next assume we are in  situation (2).  Since $P$ has finite index in $M$ then by \cite[Lemma 2.4]{CIK13} we have   $pMp\preceq^s_{pMp} P$ and thus $\Delta(pMp)$ is amenable relative to $\Delta(P)$ inside $pMp\bar\otimes L\La$.
By \cite[Theorem 2.4]{OP07} this further implies that $\Delta(pMp)$ is amenable relative to $M\otimes  1$ inside $M\bar\otimes L\La$. Applying \cite[Proposition 3.5]{CIK13} we deduce that  $\delta_k(\G)$ is amenable and finitely generated. By Remark \ref{vab} $\delta_k(\G)$ is virtually abelian.

Now we show the induction step. For simplicity we denote by $A_n=A^k_{n_k}$, $B_n=B^k_{n_k},$ $\La_{n-1}=\La^k_{n_k-1}$ and following the notation a) above we either have  $\G=\La_{n-1}\star_{A_n}B_n$ or $\G=\La_{n-1}\star_{A_n}$. Applying \cite[Theorem 1.6]{Io12} (or more directly \cite[Theorem A]{Va13}) and \cite[Theorem 4.1]{Va13} for subalgebra $\Delta(C)\subset M\bar \otimes L\G= (M\bar\otimes L \La_{n-1}) \star_{M\bar\otimes LA_n} (M\bar\otimes LB_n) $, one of the following must hold:
\begin{enumerate}
\item[(3)] $\Delta(C)\preceq_{M\bar\otimes L(\G)} M\bar\otimes L(A_n)$;
\item[(4)] $\Delta(P)\preceq_{M\bar\otimes L(\G)} M\bar\otimes  L(B_n)$;
\item[(5)] $\Delta(P)\preceq_{M\bar\otimes L(\G)} M\bar\otimes  L(\Lambda_{n-1})$;
\item[(6)] $\Delta(P)$ is amenable relative to $M\bar \otimes L(A_n)$ inside $M\bar \otimes L(\G)$.
\end{enumerate} 

We will analyze each of these cases separately. First, assume we are in situation (6). Since $P$ has finite index in $M$ then \cite[Lemma 2.4]{CIK13} implies  that $pMp\preceq^s_{pMp} P$ and thus $\Delta(pMp)$ is amenable relative to $\Delta(P)$ inside $pMp\bar\otimes L(A_n)$. Hence \cite[Theorem 2.4]{OP07} gives that $\Delta(pMp)$ is amenable relative to  $pMp\bar\otimes L(A_n)$ inside $M\bar \otimes L(\La)$. Applying Proposition \ref{rel-am} ( whose proof is similar with the proof of \cite[Proposition 3.5]{CIK13}) this further implies that $\delta_k(\G)$ is amenable relative to $A_n$ inside $\La$. Since $A_n$ is finitely generated abelian it follows that $\delta_k(\G)$ is finitely generated amenable and hence virtually abelian, as before.

Assume (4). Then by Proposition \ref{intpreim} we have $P\preceq_M A\rtimes  \delta_k^{-1}( \mu B_n \mu^{-1})$ where $\mu\in \La_{n-1}\star_{A_n} B_n=H_k$. Writing $\La_{n-1}\star_{A_n} B_n=(\mu \La_{n-1}\mu^{-1})\star_{(\mu A_n \mu^{-1})} (\mu B_n\mu^{-1})$ can assume w.l.o.g.\ that $\mu B_n \mu^{-1}=B_n$ and proceeding as in the beginning of the proof of \cite[Theorem 3.1]{CIK13} we can find projections $p_1, p_2\in \mathcal Z(P)$ with $p_1+p_2=p$ and $p_1\neq 0$ such that  $Pq_1 \preceq^s_M A\rtimes  \delta_k^{-1}(B_n)$ and $Pp_2 \npreceq_M A\rtimes  \delta_k^{-1}(B_n)$. As before, since $ P\subset pMp$  has finite index it follows that  $p_1Mp_1  \preceq^s_M Pp_1$. Altogether, by \cite[Remark 3.7]{Va07}, we have  $p_1Mp_1  \preceq_M A\rtimes  \delta_k^{-1}(B_n)$. This further implies that $\delta_k^{-1}(B_n)$ has finite index in $\G$ and since $B_n$ is abelian then it follows that $\delta_k(\G)$ is virtually abelian.

Assuming (5), Proposition \ref{intpreim} implies $P\preceq_M A\rtimes  \delta_k^{-1}(\mu \Lambda_{n-1}\mu^{-1})$ for some $\mu\in \La_{n-1}\star_{A_n} B_n=H_k$. As before suppose w.l.o.g. that $\mu \Lambda_{n-1}\mu^{-1}=\La_{n-1}$. As in the previous case $\delta_k^{-1}(\La_{n-1})$ has finite index in $\G$. Also since $C\preceq_M A\rtimes  \delta_k^{-1}(\Lambda_{n-1})$ and $C\subset pMp$ is a masa then by \cite[Proposition 3.6]{CIK13} we can find a non-zero projection $p_1\in A\rtimes  \delta_k^{-1}(\Lambda_{n-1})$ and a masa $B_1\subset p_1(A\rtimes  \delta_k^{-1}(\Lambda_{n-1}))p_1$ such that $P_1\subset p_1(A\rtimes  \delta_k^{-1}(\Lambda_{n-1}))p_1$ has finite index, where $P_1=\mathcal N_{p_1(A\rtimes  \delta_k^{-1}(\Lambda_{n-1}))p_1}(B_1)''$.  Moreover, we can find non-zero projections $p_0\in C$, $p_1'\in B_1'\cap p_1Mp_1$, and a unitary $u\in M$ such that $u(Cp_0)u^*=B_1p_1'$. Also, by the induction assumptions applied to the group homomorphism $(\delta_k)_{|\delta_k^{-1}(\La_{n-1})}\colon\delta^{-1}(\La_{n-1})\ra \La_{n-1}$ and the algebra $B_1\subset p_1(A\rtimes  \delta_k^{-1}(\Lambda_{n-1}))p_1$, we have one of the following: 
\begin{enumerate}
\item[(7)] $\delta_k(\delta_k^{-1}(\La_{n-1}))$ is virtually abelian;
\item[(8)] $B_1\preceq^s_ {A\rtimes  \delta_k^{-1}(\Lambda_{n-1})} A\rtimes  \ker(\delta_{|\delta_k^{-1}(\La_{n-1})})$ .
\end{enumerate}
Since $\delta_k^{-1}(\La_{n-1})$ has finite index in $\G$ then (7)  implies that $\delta_k(\G)$ is also virtually abelian. Also if we assume (8) then combining with the above we get that $C\preceq_M A\rtimes \ker ( \delta_k)$.

 Assuming (3), Proposition \ref{intpreim} gives $C\preceq_M A\rtimes  \delta_k^{-1}(\mu A_n\mu^{-1} )$ for some $\mu\in \La_{n-1}\star_{A_n} B_n=H_k$. Since $H_k$ is CSA, one can find a maximal abelian subgroup $\mu A_n \mu^{-1}< \Sigma<H_k$ which is malnormal in $H_k$ and notice  that $C\preceq_M A\rtimes  \delta_k^{-1}(\Sigma)=N$. Thus, one can find non-zero projections $z\in C$ and $q\in N$, a non-zero partial isometry $v\in Mp$, and a $*$-homomorphism $\theta \colon Cz\ra qNq$ such that 
 \begin{equation}\label{8.1}\theta(x)v=vx,\text{ for all }x\in Cz.\end{equation}   
 
 Since $C\subset pMp$ is a masa then $v^*v =z$ and $q':=vv^*\in \theta (Cz)'\cap qMq$. Also, we can assume w.l.o.g.\ that the support projection of $E_N(q')$ equals $q$.  Moreover, as in \cite[Lemma 1.5]{Io11} we can assume that $\theta(Cz)$ is a masa in $qNq$.  Notice that if $\theta(Cz)\preceq_N A\rtimes \ker(\delta_k)$ then using \cite[Lemma 1.4.5]{IPP05} we get that $C\preceq_M A\rtimes \ker(\delta_k)$. So for the rest assume that $\theta(Cz)\npreceq_N A\rtimes \ker(\delta_k)$. Applying Proposition  \ref{int} we have $\mathcal N_{qMq}( \theta (Cz))''\subset qNq$. Since $\theta(Cz)$ is a masa in $qNq$ this further implies $vv^*=q'\in qNq\cap \theta (Cz)'= \theta(Cz)$. If $u\in M$ is unitary such that $v= uz$ then $uzu^*=q'$ and relation (\ref{8.1}) implies that  $ \theta(Cz)q' =uBzu^*$. This together with the hypothesis assumptions imply that  $\mathcal N_{q'Mq'}( \theta (Cz)q')''=  u\mathcal N_{zMz}(Cz)''u^*$ has finite index in $uzMzu^*$. Thus, since $\mathcal N_{qMq}( \theta (Cz))''\subset qNq$ then $q'Nq'$ also has finite index in $q'Mq'$. Since by assumptions these are II$_1$ factors if follows that $\delta_k^{-1}(\Sigma)$ has finite index in $\G$; hence, $\delta_k(\G)$ is virtually abelian.
 
 Summarizing, in all cases we have obtained that either $C\preceq_M A\rtimes \ker ( \delta_k)$ or $\delta_k(\G)$ is virtually abelian. Proceeding as in the beginning of the proof of \cite[Theorem 3.1]{CIK13} we obtain the desired conclusion. \end{proof}

\begin{theorem}\label{c-sr1}Let $H_1, H_2, \ldots, H_n$ be groups as in Notation \ref{not2}. Let $\G$ be an ICC group for which exists in injective group homomorphism $\phi\colon \G_0\ra H_1\times H_2\times \cdots \times H_n$. Assume that $\G_0\ca (X,\mu)$ is  free, measure preserving action on a non-atomic probability space whose restriction to any non-amenable subgroup of $\G_0$ is ergodic. Denote by $M_0=L^{\infty}(X)\rtimes \G_0$ the corresponding group measure space von Neumann algebra. For any Cartan subalgebra $B\subset M_0$ there exists a unitary $d\in \mathcal U(M_0)$ such that $dBd^*=L^\infty(X)$.     
\end{theorem}

\begin{proof}  For a subset $F\subset\{1,2,\ldots,n\}$ we denote by $H_F$ the subgroup  consisting of all elements of $H_1\times H_2\times \cdots \times H_n$ whose $i^{th}$ coordinate is trivial, for all $i\in \{1,2,\ldots,n\}\setminus F$. Let $1\leq m\leq  n$ be the smallest integer satisfying the following properties: \begin{enumerate}
\item there exists a finite index subgroup of $\G <\G_0$ for which there exists an injective group homomorphism $\theta \colon \G\ra H_F$ for some $F\subset \{1,2,\ldots,n\}$ with $|F|=m$;
\item there is no finite index subgroup $\G_1< \G$ for which there exists an injective group homomorphism $\nu \colon \G_1\ra H_K$ for any $K\subset \{1,2,\ldots,n\}$ with $|K|=m-1$.
\end{enumerate} 
Assume w.l.o.g.\ that $F=\{1,2,\ldots,m\}$ so $H_F=H_1\times H_2\times \cdots\times H_m$. Denote by $A=L^\infty(X)$ and since  $\G$ has finite index in $\G_0$ then we have that $B\preceq_{M_0} A\rtimes \G=M$. Thus there exists non-zero projections $r\in B$, $p\in M$, a partial isometry $w\in M$ and a $*$-homomorphism  $\theta_0\colon Br\ra pMp$ such that $\theta_0(x)w =wx$, for all $x\in Br$. Notice that $w^*w=r$ and $ww^*\in \theta_0(Br)'\cap pM_0p$. Also we can assume w.l.o.g.\ that the central support of $E_M(ww^*)$ equals $p$. Since $B\subset M_0$ is a Cartan subalgebra then by  \cite[Proposition 3.6]{CIK13} (see also \cite[Lemma 1.5]{Io11}) we can assume that $C:=\theta_0(Br)$ is also a masa in $pMp$ and its normalizer $P=\mathcal N_{pMp}(C)''$ has finite index in $pMp$.
For all $k=1,\ldots,m$ denote by $\delta_k=\pi_{H_k}\circ \theta$, where $\pi_{H_k}\colon H_1\times \cdots\times H_m\ra H_k$ is the canonical projection. Clearly $\delta_k$'s satisfy the properties listed in Notation \ref{not2}.  Then (1) above and Lemma \ref{dicho} imply that for each $1\leq k\leq m$ we have either c) $C\preceq^s_M A\rtimes \ker(\delta_k)$, or, d) $\delta_k(\G)$ is virtually abelian.

Next we claim that possibility d) above never occurs. Indeed if $m=1$ then this is straightforward as the group $\G$ is assumed non-amenable. When $m\geq 2$ assume there exists $k$ such that  $\delta_k(\G)$ is virtually abelian. Since $\G_0$ is  ICC then so is  $\G$ and by Proposition \ref{abelianreduction} there will be a finite index subgroup of $\G_1< \G$ and an injective group homomorphism $\nu \colon \G_1\ra H_K$ where $K=\{ 1,2,\ldots,m\}\setminus \{k\}$. Since $|K|=m-1$, this contradicts the minimality of $m$. Thus for every$1\leq k\leq m$ we have c) $C\preceq^s_M A\rtimes \ker(\delta_k)$. 
 
Notice that by similar arguments as before we can assume that c) holds if we replace $C$ by any algebra $Ca$ where $a\in C$ is a projection.  Moreover, since $C$ is a masa in $pMp$, $\ker(\delta_k)$ are normal in $\G$ satisfying $\bigcap^m_{k=1} \ker(\delta_k)=\{ e\}$, then \cite[Lemmas 2.3 and 2.7]{Va10} further imply that $C\preceq_M A$.  Combining this with the first part of the proof we have by \cite[Lemma 1.4.5]{IPP05} that  $B \preceq_M A$; then using \cite[Appendix 1]{Po01} we get the desired conclusion.\end{proof}

\subsection{Residually free groups and residually $\La$-groups}\label{prelimrhg}

Let $\La$ be a group.
Suppose that $\lambda$ is an element of $\Lambda$ whose centralizer in $\La$, denoted by $C=C_\La(\la)$, is maximal abelian in $\La$.
The amalgamated free product $\La \star_C (C\times \mathbb{Z})$ is then called an \textit{extension of centralizers} of $\La$.
A group obtained from $\La$ by a finite sequence of extensions of centralizers is called an \textit{iterated extension of centralizers} of $\La$.

A group $H$ is called \textit{toral relatively hyperbolic} if $H$ is torsion-free and hyperbolic relative to some collection of subgroups, $\{ P_1,\ldots , P_s\}$, such that each $P_i$ is finitely-generated abelian.
Any toral relatively hyperbolic group is CSA (\cite[Lemma 6.9]{Gr09}, \cite[Section 1.4]{KM12}).
It implies that in any toral relatively hyperbolic group, the centralizer of any non-neutral element is maximal abelian (\cite[Remark 4]{MR96}).

Toral relative hyperbolicity is preserved under taking an extension of centralizers.
In fact, let $\La$ be a toral relatively hyperbolic group and pick a non-neutral element $\lambda \in \La$.
The centralizer $C=C_\La(\la)$ is maximal abelian in $\La$.
Let $L =\La \star_C (C\times \mathbb{Z})$ be the extension of centralizers of $\La$.
By \cite[Theorem 0.1 (2)]{Da03}, $L$ is toral relatively hyperbolic.

A group $\G$ is called \textit{residually free} if for every $\g\in \G$ there exists a homomorphism $\phi \colon \G \ra \mathbb{F}_2$ such that $\phi(\g)\neq 1$.
A group $\G$ is called \textit{fully residually free} if for every finite subset $F\subset \G$ there exists a homomorphism $\phi \colon \G \ra \mathbb{F}_2$ such that $\phi(\g)\neq 1$, for all $\g \in F$.
It is well known that the class of finitely generated, fully residually free groups coincides with the class of limit groups in the sense of Sela, \cite[Theorem 4.6]{Se01} (see also \cite[Theorem 1.1]{CG05}).
We refer to \cite{CG05, KM98a, KM98b, Se01} for details on (fully) residually free groups.



\begin{cor}\label{c-sr}  If  $\G$ is any finitely generated, ICC, residually free group, then any free, mixing, p.m.p.\ action $\G \ca (X,\mu)$ on a non-atomic probability space is $\mathcal C$-superrigid.  \end{cor}

\begin{proof}
It follows from \cite[Corollary 2]{KM98b} and \cite[Claim 7.5]{Se01} that there exist finitely generated, fully residually free groups $G_1,\ldots, G_s$ such that $\G$ is a subgroup of the direct product $G_1\times \cdots \times G_s$.
It further follows from \cite[Theorem 4]{KM98b} that each $G_i$ is a subgroup of an iterated extension $H_i$ of centralizers of a finitely generated, free group (see also \cite[Theorem 4.2]{CG05}).
Notice that any iterated extension of centralizers of $\mathbb{Z}$ is $\mathbb{Z}^n$ for some $n\in \mathbb{N}$, and is a subgroup of an iterated extension of centralizers of $\mathbb{F}_2$.
We may therefore assume that each $H_i$ is an iterated extension of centralizers of a finitely generated, non-abelian free group $F_i$, and is obtained from $F_i$ by finite iterated steps of amalgamation of the type described in Notation \ref{not2}.
The corollary follows from Theorem \ref{c-sr1}.
\end{proof}

Fix a group $\La$. A group $\G$ is called \emph{residually-$\La$} if for every $\g\in \G$ there exists a homomorphism $\phi \colon \G \ra \La $ such that $\phi(\g)\neq 1$.
When we have the stronger property that for every finite set  $F\subset  \G$ there exists a homomorphism $\phi \colon \G \ra \La $ such that $\phi(\g)\neq 1$, for all $\g\in F$,  then $\G$ is called \emph{fully residually-$\La$}.
We refer to \cite{Gr05, Gr09, KM13, KM12} for details on (fully) residually $\Lambda$-groups when $\La$ is hyperbolic, or more generally toral relatively hyperbolic.



\begin{cor}\label{c-sr-La} Let $\La$ be a torsion-free, non-elementary, hyperbolic group. If $\G$ is any finitely presented, ICC, residually $\La$-group, then any free, mixing, p.m.p.\ action $\G \ca (X,\mu)$ on a non-atomic probability space is $\mathcal C$-superrigid.  \end{cor}

\begin{proof}
It follows from \cite[Theorem 3.21]{KM13} that there exist iterated extensions of centralizers of $\La$, $H_1,\ldots, H_s$, such that $\G$ is a subgroup of the direct product $H_1\times \cdots \times H_s$.
The results mentioned in the beginning of this subsection imply that each $H_i$ is toral relatively hyperbolic and is obtained from $\La$ by finite iterated steps of amalgamation of the type described in Notation \ref{not2}.
The corollary follows from Theorem \ref{c-sr1}.
\end{proof}

\end{document}